\documentclass[12pt]{amsart}
\usepackage{mathscinet}
\usepackage{amsmath,amsthm,amsfonts,amssymb,amscd,euscript}
\usepackage{color}
\usepackage{graphics}
\usepackage{tikz}
\input{xy}
\xyoption{all}
\usepackage{hyperref}

\newlength{\defbaselineskip}
\theoremstyle{plain}
\newtheorem{theorem}{Theorem}[section]

\newtheorem{proposition}[theorem]{Proposition}
\newtheorem{corollary}[theorem]{Corollary}
\newtheorem{lemma}[theorem]{Lemma}

\theoremstyle{definition}
\newtheorem{example}[theorem]{Example}
\newtheorem{definition}[theorem]{Definition}

\newtheorem{remark}[theorem]{Remark}

\newtheorem{conjecture}[theorem]{Conjecture}

\newcommand{\supp}{\operatorname{supp}}

\newcommand{\R}{\mathbb{R}}

\newcommand{\ff}{F}
\newcommand{\g}{\alpha}
\newcommand{\GG}{G}

\newcommand{\epell}{\varepsilon}

\newcommand{\lowerE}{{\textrm{u}_E}}
\newcommand{\upperE}{{\textrm{v}_E}}
\newcommand{\lowerF}{{\textrm{u}_F}}
\newcommand{\upperF}{{\textrm{v}_F}}

\theoremstyle{plain}
\newtheorem{theoremalpha}{Theorem}

\newtheorem{conjecturealpha}[theoremalpha]{Conjecture}

\theoremstyle{plain}

\theoremstyle{definition}

\numberwithin{equation}{section}

\newtheorem*{rep@theorem}{\rep@title} \newcommand{\newreptheorem}[2]{%
\newenvironment{rep#1}[1]{%
\def\rep@title{\bf #2 \ref{##1}}%
\begin{rep@theorem} }%
{\end{rep@theorem} } }
\makeatother
\newreptheorem{theorem}{Theorem}
\newreptheorem{conjecture}{Conjecture}

\begin{document}

\title{On the two-dimensional Jacobian conjecture: Magnus' formula revisited, II}

\author{Jacob Glidewell}
\address{Department of Mathematics, University of Alabama, Tuscaloosa, AL 35487, U.S.A.}
\email{jbglidewell@crimson.ua.edu}

\author{William E. Hurst}
\address{Department of Mathematics, University of Alabama, Tuscaloosa, AL 35487, U.S.A.}
\email{wehurst@crimson.ua.edu}

\author{Kyungyong Lee}
\address{Department of Mathematics, University of Alabama,
	Tuscaloosa, AL 35487, U.S.A. 
	and Korea Institute for Advanced Study, Seoul 02455, Republic of Korea}
\email{kyungyong.lee@ua.edu; klee1@kias.re.kr}

\author{Li Li}
\address{Department of Mathematics and Statistics,
Oakland University, 
Rochester, MI 48309, U.S.A.}
\email{li2345@oakland.edu}

\begin{abstract}
This article is part of an ongoing investigation of the two-dimensional Jacobian conjecture. In the first paper of this series, we proved the generalized Magnus' formula. In this paper, inspired by cluster algebras, we introduce a sequence of new conjectures including the remainder vanishing conjecture. This makes the generalized Magnus' formula become a useful tool to show the two-dimensional Jacobian conjecture. In the forthcoming paper(s), we plan to prove the remainder vanishing conjecture. 

\end{abstract}

\thanks{This paper grew out of an undergraduate research project for JG and WEH  at the University of Alabama. KL was supported by the University of Alabama, Korea Institute for Advanced Study, and the NSF grant DMS 2042786.} 

\maketitle

\section{introduction}\label{section_intro}
The Jacobian conjecture, raised by Keller \cite{Keller}, has been studied by many mathematicians: a partial list of related results includes \cite{AM,A,AdEs,ApOn,BCW,BK,DEZ,dBY,NVC,CW1,CW2,CZ,Dru,EssenTutaj,EssenWZ,Gwo,Hub,JZ,Kire,LM,M,MU,MO,Nagata,NaBa,Wang,Yag,Yu}.  A survey is given in \cite{Essen,vdEssen}. In this series of papers we exclusively deal with the plane case. Hence whenever we write the Jacobian conjecture, we mean the two-dimensional Jacobian conjecture.

Let $\mathcal{R}=\mathbb{C}[x,y]$. For simplicity, let $[f, g] :=\det \begin{pmatrix}{\partial f}/{\partial x} & {\partial g}/{\partial x}\\ {\partial f}/{\partial y} &{\partial g}/{\partial y}\end{pmatrix}\in \mathcal{R}$ for any pair of polynomials $f,g\in \mathcal{R}$. %Similarly $[f,g]$ is defined for $f,g\in K[x,y][[t]]=K[[t]][x,y]$. 

\noindent \textbf{Jacobian conjecture.} 
\emph{Let $f,g\in \mathcal{R}$. Consider the  polynomial map }$\pi: \mathcal{R}\longrightarrow \mathcal{R}$\emph{ given by }
$\pi(x)=f$\emph{ and }$\pi(y)=g$. \emph{If }$[f,g]\in \mathbb{C}\setminus\{0\}$ (in which case $(f,g)$ is called a \textit{Jacobian pair}), \emph{then }$\pi$\emph{ is bijective.}

A useful tool to study this conjecture is the Newton polygon. One source for this is \cite{CW2}, but we redefine it here. Let $f=\displaystyle\sum_{i,j\geq 0} f_{ij}x^iy^j$ be a polynomial in $\mathcal{R}$ with $f_{ij}\in \mathbb{C}$. The \textit{support} of $f$ is defined as 
\[\supp(f)=\{(i,j)\in \mathbb{Z}^2 \mid f_{ij}\neq 0\} \subseteq \R^2.\]
Define the \textit{Newton polygon} for $f\in\mathcal{R}$, denoted $N(f)$,  to be the convex hull of $\supp(f)$.  %The support and Newton polygon of a Laurent polynomial in $K[x^{\pm 1},y^{\pm 1}]$ are similarly defined.
The \textit{augmented Newton polygon} for $f\in\mathcal{R}\setminus\{0\}$, denoted $N^0(f)$,  is the convex hull of $\supp(f)\cup \{(0,0)\}$.\footnote{This is the definition of Newton polygon used in some literatures on Jacobian conjecture, for example \cite{CW2}.} If $f=0$ then let $N^0(f)=\emptyset$. A vertex of $N^0(f)$ is called a nontrivial vertex if it is not equal to $(0,0)$. 
It is clear that 
 $N(f)\subseteq N^0(f)\subseteq \mathbb{R}_{\ge0}^2$ for $f\in \mathcal{R}$. 
 
We denote by $\deg(f)$ the total degree of a polynomial $f$.  Recall that if $(f,g)$ is a Jacobian pair with $\deg(f)>1$ and $\deg(g)>1$, then $N^0(f)$ is similar to $N^0(g)$ with the origin as center of similarity  \cite[Theorem 10.2.1]{vdEssen}. 
 A result of Abhyankar (for instance, see \cite[Theorem 10.2.23 ``3)$\Leftrightarrow$4)'']{vdEssen}) asserts that the Jacobian conjecture is equivalent to the statement that 
``For an arbitrary Jacobian pair $(f,g)$, 
we must have either $\deg(f)$ divides $\deg(g)$, or $\deg(g)$ divides $\deg(f)$''.
 The latter can be restated as the following (using the fact that $[F,G]\in\mathbb{C}$ implies $N^0(G)$ is similar to $N^0(F)$):

\begin{conjecturealpha}\label{Jac_conj}
Let $a,b\in \mathbb{Z}_{>0}$ be relatively prime with $2\le a < b$. 
Suppose that $F,G\in \mathcal{R}$ satisfy the following:

\noindent\emph{(1)} $[\ff,\GG]  \in  \mathbb{C}$;

\noindent\emph{(2)} $\{(0,0),(0,1),(1,0)\}\subset N(\ff)$, and $N(\ff)$ is similar to $N(\GG)$ with
the origin as center of similarity and with ratio $\deg(\ff) : \deg(\GG) = a : b$. 

\noindent Then  $[\ff, \GG]=0$. 
\end{conjecturealpha}

In this paper, inspired by cluster algebras, we introduce a sequence of new conjectures which implies Conjecture~\ref{Jac_conj}. %We give an outline for our proposed proof of the conjecture, using the generalized Magnus' formula \cite{}.    

\noindent\emph{Acknowledgements.} We would like to thank David Wright for valuable discussion, and Christian Valqui for numerous helpful suggestions. We also thank Rob Lazarsfeld, Lenny Makar-Limanov, and Avinash Sathaye for their correspondences.

\section{New conjectures}

Let $W=\{(0,1),(1,1),(1,0)\}\subseteq \mathbb{Z}^2$. An element  $w = (u, v) \in W$ is called a \emph{direction}. To each such a direction we consider its $w$-grading on $\mathcal{R}$.
So $\mathcal{R} = \oplus_{n\in \mathbb{Z}} \mathcal{R}^w_n$, where $\mathcal{R}^w_n=\mathcal{R}^{(u,v)}_n$ (or simply denoted by $\mathcal{R}_n$ if $w$ is clear from the context) is the $\mathbb{C}$-vector space generated by the monomials $x^iy^j$ with
$ui + vj = n$. A non-zero element $f$ of $\mathcal{R}_n$ is called a $w$-homogeneous element of $\mathcal{R}$
and $n$ is called its $w$-degree, denoted $w\text{-}\deg(f)$. 
For convenience of notation, we sometimes denote $(1,1)\text{-}\deg(f)$ simply by $\deg(f)$. 
The element of highest $w$-degree in the homogeneous decomposition of
a non-zero polynomial $f$ is called its $w$-leading form and is denoted by $f_+$. The
$w$-degree of $f$ is by definition $w\text{-}\deg(f_+)$. Here we mainly use $w=(0,1)$ or $(1,1)$, but $(1,0)$-degrees will briefly appear in the proof of Lemma~\ref{Div_to_limit1}.

For convenience, we say that $g$ is divisible by $f$, denoted $f|g$, if $g/f$ is a (Laurent) polynomial. For example, $(x+1)^{i-1}|(x+1)^i$ for all $i\in\mathbb{Z}$. This way we do not need to worry about negative exponents. 

For a  line segment $\overline{AB}\subset \mathbb{R}^2$ whose endpoints are both in $\mathbb{Z}^2\subset  \mathbb{R}^2$, we define the length ${\rm len}(\overline{AB})\in \mathbb{Z}_{\ge 0}$ to be one less than the number of lattice points on $\overline{AB}$. 
For any direction $w=(u,v)\in W$ and for any nonzero $w$-homogeneous polynomial $h\in \mathcal{R}$, we define ${\rm len}(h)$ to be the length of $N(h)$; that is, if $h= a_0x^by^c+a_1x^{b+v}y^{c-u}+a_2x^{b+2v}y^{c-2u}+\cdots+a_lx^{b+lv}y^{c-lu}$ with $a_0\neq0$ and $a_l\neq0$, then ${\rm len}(h)=l$.  If $h=0$ then we define ${\rm len}(h)$ to be equal to $-\infty$. The following lemma is obvious, but will be very useful in this paper.

\begin{lemma}\label{elem_zero} Fix $n,l \in \mathbb{Z}$. 

\noindent\emph{(1)} If $h\in\mathcal{R}^{(0,1)}_n$ is divisible by $(x+1)^l$ and ${\rm len}(h)<l$, then $h=0$. 

\noindent\emph{(2)} If $h\in\mathcal{R}^{(1,1)}_n$ is divisible by $(x+y)^l$ and ${\rm len}(h)<l$, then $h=0$. 
\end{lemma}

For any $m,n\in \mathbb{Z}$, we define the trapezoid $T_{m,n}$ by 
$$T_{m,n}:=\{(x,y)\in \mathbb{R}^2 \, : \,  0\le y \le n\text{ and }0\le x\le m+n-y\}.$$ 
For a picture, see the first trapezoid in Figure \ref{fig:NNN}.
Note that if $n<0$ then $T_{m,n}=\emptyset$.

Define 
$$\mathcal{Q}=\{(a,b,m,n)\in \mathbb{Z}_{>0}^4\, : \,  m<n, \ a|m,\ a|n, \  \gcd(a,b)=1\text{ and }2\le a < b\}.$$
For any $f\in \mathcal{R}$ and any  $w = (u, 1) \in W$, we write the $w$-homogeneous decomposition $f=\sum_i f^{w}_i$ where $f^{w}_i=f^{(u,1)}_i\in \mathcal{R}^w_i$. The following Conjecture \ref{Jac_conj2} is inspired from divisibility conditions for greedy basis elements of rank 2 cluster algebras (\cite[Theorem 2.2]{LLS},\cite[Theorem 4]{LLRZ}). We shall show that Conjecture \ref{Jac_conj2} implies Conjecture~\ref{Jac_conj} (see Lemma~\ref{BtoA}), hence implies the Jacobian conjecture. 

\begin{conjecturealpha}\label{Jac_conj2}
Let $(a, b, m, n) \in \mathcal{Q}$. 
Suppose that $\ff,\GG\in \mathbb{C}[x,y]$ satisfy the following:

\noindent\emph{(1)} $N^0(\ff)=T_{m,n}$ and $N^0(\ff - x^m y^n)\subsetneq N^0(\ff)$;

\noindent\emph{(2)} $F^{(0,1)}_{n-i}$ is divisible by $(x+1)^{m-i}$ for $0\le i\le m$; 

\noindent\emph{(3)} $F^{(1,1)}_{m+n-i}$ is divisible by $(x+y)^{n-i}$ for $0\le i\le n$;

\noindent\emph{(4)} $N^0(\GG)=T_{bm/a,bn/a}$ and $N^0(\GG - x^{bm/a} y^{bn/a})\subsetneq N^0(\GG)$; 

\noindent\emph{(5)} $[\ff,\GG]  \in  \mathbb{C}$.

\noindent Then  $[\ff, \GG]=0$. 
\end{conjecturealpha}

\begin{remark}
It can be shown that the condition(s) (2) and/or (3) is equivalent to saying that there exists an automorphism  $\xi$ of $\mathbb{C}[x,y]$ such that $N^0(\xi(F))$ is contained in a smaller trapezoid. For a more precise statement and its proof, see \cite[Proposition 3.9]{GGV} or Lemma~\ref{lemma:N(F) change variables}.
One may wonder why we formulated conjecture~\ref{Jac_conj2} the way we did here instead of using known ``smaller" trapezoidal shapes (for instance, as given in \cite{CN}). The reason is implicitly given in \cite{HLLN}, where we associated divisibility conditions with the geometry of Newton polygons.  Basically if we consider divisibility by (a power of) a binomial such as $x+1$ or $x+y$, then Lemma~\ref{elem_zero} will be very useful. If we apply the generalized Magnus' formula to a smaller trapezoid, we are forced to divide a polynomial by a monomial, which does not help our argument. 
%in our approach the Newton polytope is not minimal. If we make it smaller then the method no longer work.]
\end{remark}

For convenience, we will use upper letters $\ff,\GG$ to denote polynomials satisfying partial or all conditions (1)--(5) of Conjecture \ref{Jac_conj2}, and use lower letters $f,g$ to denote polynomials in general. 

It is easy to see that if $[F,G]  \in  \mathbb{C}$ and $F=\alpha(E)$ for some $E\in \mathcal{R}\setminus \mathbb{C}$ and some $\alpha(z)\in \mathbb{C}[z]$ with $\deg(\alpha)\ge2$, then $[\ff,\GG]=0$ (see Lemma~\ref{CtoB}). In this series of papers, we plan to prove Conjecture~\ref{Jac_conj2} by constructing such $E$ and $\alpha$.

For a real number $r\in \mathbb{R}$ and a subset $S\subseteq\mathbb{R}^2$, denote $rS:=\{rs\ : s\in S\}\subseteq\mathbb{R}^2$.  

For $\ff$ satisfying the condition (1) in Conjecture \ref{Jac_conj2}, denote $C=(m,n)\in \mathbb{R}^2$, $\mathcal{N}'=\frac{1}{a} N^0(\ff)=\frac{1}{a}T_{m,n}$, and $\mathcal{N}''=\mathcal{N}'+\frac{a-1}{a}\overrightarrow{OC}$.
See Figure \ref{fig:example} for examples of $a=2$ and $a=4$.

\begin{figure}[h]
\begin{center}
\begin{tikzpicture}[scale=0.50]
\usetikzlibrary{patterns}
\draw (0,8)--(4,8)--(12,0);
\draw (0,0)--(0,4)--(2,4)--(6,0)--(0,0);
\fill[black!10] (0,0)--(0,4)--(2,4)--(6,0)--(0,0);
\draw (1,2) node[anchor=west] {\small $\mathcal{N}'$};
\draw (2,4)--(2,8)--(4,8)--(8,4)--(2,4);
\fill[blue!10] (2,4)--(2,8)--(4,8)--(8,4)--(2,4);
\draw (3,6) node[anchor=west] {\small $\mathcal{N}''$};
\draw (4,8) node {\huge .};
\draw (4, 8) node[anchor=south west] {\small $C$};
\draw[->] (0,0) -- (13,0)
node[above] {\tiny $x$};
\draw[->] (0,0) -- (0,9)
node[left] {\tiny $y$};

\begin{scope}[shift={(15,0)}]
\usetikzlibrary{patterns}
\draw (0,8)--(4,8)--(12,0);
\draw (0,0)--(0,2)--(1,2)--(3,0)--(0,0);
\fill[black!10] (0,0)--(0,2)--(1,2)--(3,0)--(0,0);
\draw (.2,1) node[anchor=west] {\small $\mathcal{N}'$};
\draw (3,6)--(3,8)--(4,8)--(6,6)--(3,6);
\fill[blue!10] (3,6)--(3,8)--(4,8)--(6,6)--(3,6);
\draw (3,7) node[anchor=west] {\small $\mathcal{N}''$};
\draw (4,8) node {\huge .};
\draw (4, 8) node[anchor=south west] {\small $C$};
\draw[->] (0,0) -- (13,0)
node[above] {\tiny $x$};
\draw[->] (0,0) -- (0,9)
node[left] {\tiny $y$};
\end{scope}
\end{tikzpicture}
\end{center}
\caption{$N^0(F)$ together with $\mathcal{N}'=N^0(Q)$ and $\mathcal{N}''$ for  $a=2$ (Left) and $a=4$ (Right).}
\label{fig:example}
\end{figure}
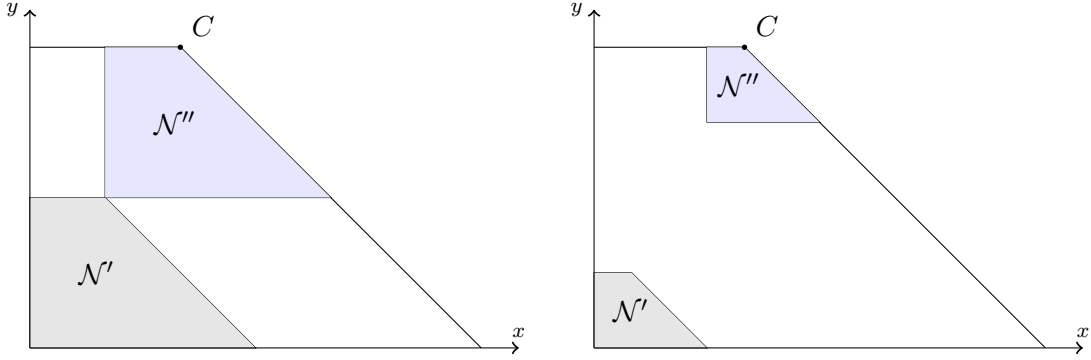

We need the following lemma, whose proof is similar to \cite[Lemma 2.6]{HLLN}.
\begin{lemma}\label{lem:c=p^2 most generalized}
Let $\ff$ satisfy the conditions {\rm(1)--(3)} in Conjecture \ref{Jac_conj2}.   

(i) There exists a unique polynomial $Q\in \mathcal{R}$ such that $N^0(Q - x^{m/a} y^{n/a})\subsetneq N^0(Q)=\mathcal{N}'$ and ${\rm supp}(\ff-Q^a)\subseteq N^0(\ff)\setminus \mathcal{N}''$.

(ii) For $Q$ obtained from (i), we have $$Q^{(0,1)}_{n/a}=(x+1)^{m/a}y^{n/a}\quad \text{and}\quad Q^{(1,1)}_{(m+n)/a} = x^{m/a}(x+y)^{n/a}.$$
\end{lemma}
\begin{proof}
%$Q=\sum_{(i,j)\in \mathcal{N}'} q_{ij}x^iy^j$
Fix two positive numbers $r_1\in\mathbb{Q}$ and $r_2\in \mathbb{R}\setminus\mathbb{Q}$ such that $\{(x,y)\ | r_1 x+r_2 y=r_1 m+r_2 n\}\cap N^0(F)=\{C\}$. Then $N^0(F)$ lies in the half plane $r_1 x+r_2 y\le r_1 m+ r_2 n$. Let $n'$ be the number of lattice points in $\mathcal{N}'$. Arrange lattice points $\{z_i=(x_i,y_i)\}_{1\le i\le n'}$ in $\mathcal{N}'$ such that $r_1 x_1+r_2 y_1>r_1 x_2+r_2 y_2>\cdots>r_1 x_{n'}+r_2 y_{n'}$. Then $z_1=\frac{1}{a}C=(m/a,n/a)$. For $i=1,\dots,n'$, denote ${\bf x}^{z_i}=x^{x_i}y^{y_i}$. % and $q_{z_i}=q_{x_iy_i}$, thus $Q=\sum_{z_i} q_{z_i}{\bf x}^{z_i}$ (note that $Q$ cannot contain any monomials that are not in $\mathcal{N}'$, otherwise ${\rm supp}(\ff-Q^a)\not\subset N(F)$ which does not satisfy the requirement). 

We inductively define the sequence of complex numbers $q_1,...,q_{n'}$ as follows. As a base step, let $q_{1}=1$. As an inductive step, 
assume $q_{1},\dots,q_{k-1}$ are already uniquely determined for some $k>1$. 
Note that $(a-1)z_1+z_k\in\mathcal{N}''$.
The next entry $q_k$ will be determined by the condition that the coefficient of ${\bf x}^{(a-1)z_1+z_k}$ in $F$ is the same as the coefficient of ${\bf x}^{(a-1)z_1+z_k}$ in $(\sum_{i=1}^k q_{i}{\bf x}^{z_i})^a$. Let  $\ff=\sum_{i,j}\lambda_{(i,j)}x^iy^j$.

Since
$$\left(\sum_{i=1}^k q_{i}{\bf x}^{z_i}\right)^a=\sum q_{i_1}q_{i_2}\cdots q_{i_a}{\bf x}^{z_{i_1}+z_{i_2}+\cdots+z_{i_a}},$$
we have
$$\aligned
\lambda_{(a-1)z_1+z_k}=\sum_{z_{i_1}+\cdots+z_{i_a}=(a-1)z_1+z_k} q_{i_1}\cdots q_{i_a}&={a\choose 1}q_1^{a-1}q_{k}+\sum_{\stackrel{z_{i_1}+\cdots+z_{i_a}=(a-1)z_1+z_k}{i_1,\dots,i_a<k}}q_{i_1}\cdots q_{i_a},\endaligned$$
which  uniquely determines  $q_{k}$. Since $\mathcal{N}''=\{(a-1)z_1+z_k\ | \  1\le k\le n'\}$, once $q_1,...,q_{n'}$ are defined, the polynomial $Q=\sum_{i=1}^{n'} q_{i}{\bf x}^{z_i}$ satisfies the desired conditions in (i). The statement (ii) easily follows from construction of $q_1,...,q_{n'}$.
\iffalse
We claim that we can recursively solve $q_{1},q_{2},\dots,q_{n}$ uniquely to satisfy the condition that $\emph{supp}(\ff-Q^a)\cap \mathcal{N}''=\emptyset$. First, $q_{1}=q_{m/a,n/a}=1$. Next,
assume $q_{z_1},\dots,q_{z_{k-1}}$ are already uniquely determined for some $k>1$. Since $(a-1)z_1+z_k$ lies in $\mathcal{N}''$, the coefficient of ${\bf x}^{(a-1)z_1+z_k}$ in $F$ is the same as the one in $Q^a$. 
Since
$$Q^a=\sum q_{z_{i_1}}q_{z_{i_2}}\cdots q_{z_+{i_a}}{\bf x}^{z_{i_1}+z_{i_2}+\cdots+z_{i_a}},$$
we have the following (recall that $q_{z_1}=1$):
$$\lambda_{(a-1)z_1+z_k}=\sum_{z_{i_1}+\cdots+z_{i_a}=(a-1)z_1+z_k} q_{z_{i_1}}\cdots q_{z_{i_a}}=aq_{z_k}+\sum_{\stackrel{z_{i_1}+\cdots+z_{i_a}=(a-1)z_1+z_k}{i_1,\dots,i_a<k}}q_{z_{i_1}}\cdots q_{z_{i_a}}$$
thus $q_{z_k}$ is uniquely determined. Since $\mathcal{N}''=\{(a-1)z_1+z_k\ | \  1\le k\le n\}$, the above method solves all $q_{z_k}$, therefore $Q$, uniquely.
\fi
\end{proof}

\begin{definition}
Let $\mathbb{T}$ be the set of  Tschirnhausen polynomials, that is, 
$$\mathbb{T}=\{ \g(z)=z^k + e_{k-1} z^{k-1} + \cdots + e_0 z^0 \in \mathbb{C}[z] \, : \, k\in \mathbb{Z}_{>0}, \, e_{k-1}=0,\text{ and }e_{k-2},\cdots,e_0\in \mathbb{C} \}.$$ For a polynomial $f\in \mathcal{R}$, let $\mathcal{E}(f)=\{ E\in \mathcal{R} \, : \, f=\g(E)\text{ for some }\g(z)\in \mathbb{T}\}$. Note that $f\in \mathcal{E}(f)$.  An element  $f\in \mathcal{R}$ with $\mathcal{E}(f) = \{f\}$ is called a \emph{principal} polynomial. 
\end{definition}

\begin{example}
For $f=(3x^2+y)^{12}+(3x^2+y)^4+1$, below we show some elements in $\mathcal{E}(f)$ and the corresponding $h$:

$E=f$, $\g=z$;

$E=\omega(3x^2+y)^4$, $\g=z^3+\omega^{-1}z+1$, where $\omega^3=1$;

$E=\omega (3x^2+y)^2$, $\g=z^6+\omega^{-2}z^2+1$,  where $\omega^6=1$;

$E=\omega (3x^2+y)$, $\g=z^{12}+\omega^{-4}z^4+1$,  where $\omega^{12}=1$.

Among the above,  the polynomials $\omega(3x^2+y)$, with $\omega^{12}=1$,  are principal polynomials. 
\end{example}

Lemma~\ref{unique_principal} below will show that $\mathcal{E}(f)$ contains a unique principal polynomial up to roots of unity. Before that, we want to first prove Lemma \ref{maximal_factorization} using algebraic geometry.

We recall some simple facts from algebraic geometry. A polynomial $f\in\mathbb{C}[x,y]$ determines a homomorphism $f:\mathbb{C}[z]\to \mathbb{C}[x,y]$ sending $z$ to $f$, thus induces a morphism $\phi_f: \mathbb{A}^2\to \mathbb{A}^1$ (which is the polynomial function determined by $f$). Conversely, every morphism $\mathbb{A}^2\to \mathbb{A}^1$ is determined by a unique polynomial. 
The following are equivalent: ``$f$ is not a constant function'' $\Leftrightarrow$ ``$\phi_f$ is a non-constant morphism'' $\Leftrightarrow$  ``$\phi_f$ is surjective''. Similarly, $\g\in \mathbb{T}$ induces a proper and finite morphism $\phi_\g:\mathbb{A}^1\to\mathbb{A}^1$ with $\deg \phi_\g=\deg \g$, and $E\in \mathbb{C}[x,y]$ induces a morphism $\phi_E:\mathbb{A}^2\to \mathbb{A}^1$. The condition $f=\g(E)$ translates to  $\phi_f=\phi_\g\circ\phi_E$, which we call the \emph{factorization}  determined by $E$.%\in \mathcal{E}(f)$. 

In the following lemma, for convenience of notation, for $f\in\mathbb{C}[x,y]$ we write $\phi_f$ simply as $f$ when no confusion should occur. By ``degree of $\phi_f$'', we mean the the $(1,1)$-degree of $f$. For $\g\in\mathbb{C}[z]$, we write $\phi_\g$ as $\g$; by ``degree of $\phi_\g$'' we mean the degree of $\g$. 

Given  a non-constant morphism  $f: \mathbb{A}^2\to\mathbb{A}^1$, we consider factorizations of $f$ of the form $f=\g\circ j$ with morphisms $j:\mathbb{A}^2\to\mathbb{A}^1$ and $\g:\mathbb{A}^1\to\mathbb{A}^1$. 
We say that two factorizations $f=\g_1\circ j_1$ and $f=\g_2\circ j_2$ are equivalent if there is an isomorphism $i:\mathbb{A}^1\to\mathbb{A}^1$ such that $\g_2=\g_1\circ i$ and $j_1=i\circ j_2$. We call $\deg(\g)$ the \emph{depth} of the factorization $f=\g\circ j$. We call $f=\g\circ j$ a deepest factorization of $f$ if there exists no other factorization of $f$ with a larger depth.

\begin{lemma}\label{maximal_factorization} Given a non-constant morphism $f: \mathbb{A}^2\to\mathbb{A}^1$, there is, up to isomorphism,  a unique universal factorization $f=\g_0\circ j_0$ in the sense that: if there is another factorization $f=\g\circ j$, then there is a morphism $i: \mathbb{A}^1\to \mathbb{A}^1$ such that $\g_0=\g\circ i$ and $j=i\circ j_0$, that is, the following diagram commutes. (Equivalently, $f=\g_0\circ j_0$ is a deepest factorization of $f$.)
$$\xymatrix{
\mathbb{A}^2\ar[rd]^(.7){j_0}\ar[rrrd]^{j}\ar@/_2pc/[ddrr]_f&\\
&\mathbb{A}^1\ar[rd]^(.6){\g_0}\ar[rr]^(.4){i}&&\mathbb{A}^1\ar[ld]^(.3){\g}\\ 
&&\mathbb{A}^1
}$$
\end{lemma}
\begin{proof}
Given any two factorizations $f=\g_1\circ j_1=\g_2\circ j_2$, we can construct a new factorization $f=\g_3\circ j_3$ which is either equivalent to one of the two factorizations, or deeper than them. Note that since $f$ is non-constant, $\g_1,\g_2,j_1,j_2$ must also be non-constant. 
$$\xymatrix{
&&\tilde{X}\ar[d]^\pi\ar@/_-10pc/[ddd]^{\g_3}\\
\mathbb{A}^2\ar[rd]^(.6){j_1}\ar[rrrd]^(.6){j_2}\ar@/_2pc/[ddrr]_f\ar[rr]^(.4)p\ar[rru]^{j_3}&&\hspace{33pt}X\subseteq X_{\rm fp} \ar[ld]^(.6){\beta_1}\ar[rd]^(.6){\beta_2}\\
&\mathbb{A}^1\ar[rd]^{\g_1}&&\mathbb{A}^1\ar[ld]_{\g_2}\\ 
&&\mathbb{A}^1
}$$

Let $X_{\rm fp}$ be the fiber product of $\g_1$ and $\g_2$. Then $f$ factors through $X_{\rm fp}$, and its image in $X_{\rm fp}$ must be irreducible. Let $X$ be the irreducible component of $X_{\rm fp}$ that contains the image of $f$. Then $f$ factors through $X$ in the sense that there is a morphism $p:\mathbb{A}^2\to X$ such that $f=\g_1\circ \beta_1\circ p=\g_2\circ \beta_2\circ p$. Let $\pi:\tilde{X}\to X$ be the normalization of $X$, then $p$ factors through $\tilde{X}$ in the sense that there is a morphism $j_3:\mathbb{A}^2\to X$ such that $p=\pi\circ j_3$. Define $\g_3=\g_1\circ \beta_1\circ \pi (=\g_2\circ \beta_2\circ \pi) : \tilde{X}\to \mathbb{A}^1$. Then $f=\g_3\circ j_3$.
Note that the curve $X$ is rational, otherwise each line in $\mathbb{A}^2$ must map to a point in $X$ because there is no non-constant rational map from a rational curve to a nonrational curve
\footnote{Here is a short sketch of proof: assume $f: C_1\to C_2$ is a non-constant rational map, where $C_1$ is a rational curve and $C_2$ is a nonrational curve. Without loss of generality we may assume that $C_1$ and $C_2$ are nonsingular projective curves. So ${\rm genus}(C_1)=0$,  ${\rm genus}(C_2)\ge1$.  
By \cite[II.2.1]{Silverman}, a rational map from a smooth projective curve to a projective curve is always a morphism. 
Thus $f$ is a morphism. Applying Riemann--Hurwitz formula to $f$, we get a contradiction $-2=2{\rm genus}(\mathbb{P}^1)-2=(\deg f)(2{\rm genus}(C_2)-2)+\sum (e_P-1)\ge 0$ where $e_P\ge1$ are ramification indices.
 }.
Since $X$ is rational and affine (but possibly singular), $\tilde{X}$ is a nonsingular affine rational curve, thus $\tilde{X}\cong\mathbb{A}^1$. Moreover since $\g_1$, $\g_1'$, and $\pi$ are all proper and non-constant, the morphism $\g_3$ is proper and non-constant. 
Note that $\deg \g_3\ge \deg \g_1$ (resp.~$\deg \g_2$), and equalities hold when the factorization $f=\g_3\circ j_3$ is equivalent to $f=\g_1\circ j_1$ (resp.~$f=\g_2\circ j_2$). 

Consider the set of all equivalence classes of the factorizations of $f$. Note that the depth of any factorization is no larger than the $(1,1)$-degree of the polynomial $f$, so there exists a deepest factorization $f=\g_0\circ j_0$ (that is, $\deg(\g_0)$ is maximal). By the above argument, we see that $f=\g_0\circ j_0$ is universal. The uniqueness of the factorization follows from the universal property.  
This completes the proof. 
\end{proof}

\begin{lemma}\label{unique_principal}
Let $f$ be a non-constant polynomial in $\mathcal{R}$. Then $\mathcal{E}(f)$ contains a unique principal polynomial up to roots of unity.
\end{lemma}
\begin{proof}
As explained before Lemma~\ref{maximal_factorization}, each $E\in \mathcal{E}(f)$ determines a factorization of $f$. Two polynomials $E_1,E_2\in \mathcal{E}(f)$ determine equivalent factorizations if and only if the corresponding $\g_1$ and $\g_2$ satisfy $\phi_{\g_1}=\phi_{\g_2}\circ i$ for an isomorphism $i:\mathbb{A}^1\to\mathbb{A}^1$. Then $i(z)=az+b$ for some $a,b\in \mathbb{C}$ with $a\neq0$. Since $\g_2\in\mathbb{T}$, we can write $\alpha_2=z^k+e_{k-2}z^{k-2}+\cdots+e_0$, and $\alpha_1=(az+b)^k+e_{k-2}(az+b)^{k-2}+\cdots+e_0=a^kz^k+ka^{k-1}bz^{k-1}+\cdots$. Then $\g_1\in\mathbb{T}$ implies $a^k=1, ka^{k-1}b=0$. It follows that $a$ must be a $(\deg \g_1)$-th root of unity, and the constant $b=0$. 

By Lemma~\ref{maximal_factorization}, $E$ is a principal polynomial if and only if it corresponds to the  deepest factorization $f=\g_0\circ j_0$ which must exist and is unique. By the previous paragraph, the corresponding polynomial $E\in \mathcal{E}(f)$ is uniquely determined up to a ($\deg \g_0$)-th root of unity. 
\end{proof}

\begin{lemma}
Let $\ff$ satisfy the conditions {\rm(1)--(3)} in Conjecture \ref{Jac_conj2}, and let $Q$ be the unique polynomial such that $N^0(Q - x^{m/a} y^{n/a})\subsetneq N^0(Q)=\mathcal{N}'$ and ${\rm supp}(\ff-Q^a)\subseteq N(\ff)\setminus \mathcal{N}''$ (as in Lemma~\ref{lem:c=p^2 most generalized}). 
Then there exists a unique integer $\delta\in\mathbb{Z}_{>0}$ such that $N^0(E)=\frac{1}{\delta a} N^0(\ff) \;(=\frac{1}{\delta a} T_{m,n} )$ for any principal polynomial $E\in\mathcal{E}(Q)$. 
\end{lemma}
\begin{proof}
Let $E\in\mathcal{E}(Q)$ be a principal polynomial . Then there is $\g\in\mathbb{T}$ such that $Q=\g(E)=E^k + e_{k-2} E^{k-2} + \cdots + e_0 E$. Note that each vertex of $N^0(E^k)$ that is not the origin must not be contained in $N^0(E^i)$ for $i<k$, we conclude that $N^0(Q)=kN^0(E)$. Letting $\delta=k$, we get the existence. The uniqueness of $\delta$ is because $E$ is unique up to a root of unity. 
\end{proof}

\begin{definition}\label{def:Ecirc}
Let $E^\circ$ be the unique polynomial among the principal polynomials in $\mathcal{E}(Q)$ that contains a term $1x^{m/a\delta}y^{n/a\delta}$. Then $N^0(E^\circ - x^{m/a\delta}y^{n/a\delta})\subsetneq N^0(E^\circ)=\frac{1}{\delta a} T_{m,n}$, and
each of the other principal polynomials contain a term $\omega  x^{m/a\delta}y^{n/a\delta}$ for a root of unity $\omega\neq1$. 
We call  $E^\circ$ the $F$-generator. 
\end{definition}

\begin{lemma}
Let $\ff$ satisfy the conditions {\rm(1)--(3)} in Conjecture \ref{Jac_conj2}. Then the set
$$\mathbb{T}_F:=\{\alpha\in \mathbb{T} \, :\, \emph{supp}(\ff - \g(E^\circ))\subseteq N^0(\ff)\setminus \mathcal{N}''\}$$
 is nonempty.
 \end{lemma}
 \begin{proof}
  Since $E^\circ$ is a principal polynomial in $\mathcal{E}(Q)$, we can write 
  \begin{equation}\label{eq:QE}
  Q=\g_Q(E^\circ)=(E^\circ)^\delta+e_{\delta-2}(E^\circ)^{\delta-2}+e_{\delta-3}(E^\circ)^{\delta-3}+\cdots+e_0.
  \end{equation}
 Denote $\g_{Q}^a=(\g_Q)^a$. Then $Q^a=\g_{Q}^a(E^\circ)=(E^\circ)^{a\delta}+ae_{\delta-2}(E^\circ)^{a\delta-2}+\cdots+e_0^a$ . Then $\text{supp}(\ff - \g_Q^a(E^\circ))=
\text{supp}(\ff - Q^a)\subseteq N^0(\ff)\setminus \mathcal{N}''$ by Lemma \ref{lem:c=p^2 most generalized}. This gives a polynomial $\g_Q^a\in \mathbb{T}_F$. 
\end{proof}

\begin{definition}\label{def:VF}
Let $$\mathbb{V}_F=\{N^0(\ff - \g(E^\circ)) \subseteq \mathbb{R}^2\, :\, \g\in \mathbb{T}_F\}.$$ 
Note that $\mathbb{V}_F$ is a nonempty set of convex polygons (where $\mathbb{V}_F$ may contain the empty polygon $\emptyset$) contained in $N^0(\ff)$ with integral vertex coordinates, so it is a finite set. We view $\mathbb{V}_F$ as a poset with the partial order given by inclusion, and define the following:

(i) Fix a minimal element  $V^\circ\in \mathbb{V}_F$, in other words, if $V$ is an element of $\mathbb{V}_F$ with $V\subseteq V^\circ$, then $V= V^\circ$. 

(ii) Fix $\g^\circ\in \mathbb{T}_F$ such that $V^\circ=N^0(\ff - \g^\circ(E^\circ))$. 

(iii) Let $\ff^\circ=\ff - \g^\circ(E^\circ)$. This $\ff^\circ$ will be called the \emph{remainder}.
\end{definition}
\begin{lemma}\label{Ecirctopdegpieces}
 Let $\ff$ satisfy the conditions {\rm(1)--(3)} in Conjecture \ref{Jac_conj2}, and let $E^\circ$ be the $F$-generator. 
 Assume $\alpha\in\mathbb{T}_F$ satisfy $V^\circ=N^0(\ff - \g(E^\circ))$. Then
 
{\rm(i)}\ $(E^\circ)^{(0,1)}_{n/a\delta}=(x+1)^{m/a\delta}y^{n/a\delta}\quad \text{and}\quad (E^\circ)^{(1,1)}_{(m+n)/a\delta} = x^{m/a\delta}(x+y)^{n/a\delta}.$

{\rm(ii)}\  $\deg(\alpha)=a\delta$.
\end{lemma}
\begin{proof}
(i) By \eqref{eq:QE}, $\big((E^\circ)^{(0,1)}_{n/a\delta}\big)^\delta=Q^{(0,1)}_{n/a}=(x+1)^{m/a}y^{n/a}$, so
$(E^\circ)^{(0,1)}_{n/a\delta}$ equals a $k$-th root of unity multiplying by $(x+1)^{m/a\delta}y^{n/a\delta}$. Since $E^\circ$ contains a term $1x^{m/a\delta}y^{n/a\delta}$ by definition, we must have $(E^\circ)^{(0,1)}_{n/a\delta}=(x+1)^{m/a\delta}y^{n/a\delta}$.  The other equality is proved similarly. 

(ii) Denote $k=\deg(\alpha)$, and note that $N^0(F)=T_{m,n}$ and $N^0(\g(E^\circ))=kN^0(E^\circ)=kT_{\frac{m}{a\delta},\frac{n}{a\delta}}=T_{\frac{km}{a\delta},\frac{kn}{a\delta}}$.  If $k>a\delta$, then vertices of $T_{\frac{km}{a\delta},\frac{kn}{a\delta}}$ that are not the origin lie outside $T_{m,n}$, thus $N^0(F^\circ)=T_{\frac{km}{a\delta},\frac{kn}{a\delta}}\supseteq N^0(F)$; if $k<a\delta$, then $N^0(F^\circ)=T_{m,n}=N^0(F)$. In both cases, $supp(F^\circ)\not\subseteq N^0(F)\setminus \mathcal{N}''$. So we must have $k=a\delta$. 
\end{proof}

%\begin{example}
%TO BE ADDED 
%\end{example}

We shall show that the following Conjecture \ref{Jac_conj3} implies Conjecture~\ref{Jac_conj2} (see Lemma~\ref{CtoB}),  hence implies the Jacobian conjecture.
\begin{conjecturealpha}\label{Jac_conj3}
Assume that $\ff$ and $\GG$ satisfy the conditions {\rm(1)--(5)} in Conjecture \ref{Jac_conj2}. Then any of the following equivalent statements is true.

{\rm(i)} $\emptyset\in \mathbb{S}_F$. 

{\rm(ii)} There is $\g\in \mathbb{T}$ such that $\ff - \g(E^\circ)=0$.

{\rm(iii)} $V^\circ=\emptyset$.

{\rm(iv)} $\ff^\circ=0$.  \end{conjecturealpha}

\begin{remark}
One may compare Conjecture~\ref{Jac_conj3} with Conjecture~\ref{GGV_conj}, which can be obtained from \cite{GGV14}. Although the Laurent polynomial setting from \cite{GGV14} may also lead to the same conjectures we are aiming at, we prefer to stay only in the polynomial ring $\mathcal{R}$ and to precisely control the principal $F$-generator $E^\circ$ and the remainder $F^\circ$. 
\end{remark}

The following definition is inspired from divisibility and support conditions for greedy basis elements of rank 2 cluster algebras (\cite[Theorem 2.2]{LLS},\cite[Theorem 4]{LLRZ}).

\begin{definition}\label{def:ith divisibility}
Let $i\in \mathbb{Z}_{\ge 0}$.  Let $\epell=\frac{1}{m(n-m)}$. %$\rho_i=\frac{i}{m(n-m)}$.  
We say that $\ff^\circ$ satisfies the $i$-th collection of $(0,1)$-divisibility conditions (for short, $(0,1)$-DC) if 
\begin{center} $(F^\circ)^{(0,1)}_{j}$ is divisible by $(x+1)^{j-\lfloor i (n-m) \epell\rfloor}$ for each $j\in\mathbb{Z}$.
\end{center}

\noindent We say that $\ff^\circ$ satisfies the $i$-th collection of $(1,1)$-divisibility conditions (for short, $(1,1)$-DC) if 
\begin{center} $(F^\circ)^{(1,1)}_{j}$ is divisible by $(x+y)^{j-\lfloor (i+1) m \epell\rfloor}$ for each $j\in\mathbb{Z}$.
\end{center}

\noindent We say that $\ff^\circ$ satisfies the $i$-th collection of divisibility conditions (DC) if it satisfies the $i$-th collection of $w$-divisibility conditions for $w\in\{(0,1),(1,1)\}$. 

\noindent We say that $\ff^\circ$ satisfies the $i$-th collection of divisibility and support conditions (DSC) if it satisfies the $i$-th collection of divisibility conditions and $$\text{supp}(F^\circ)\subseteq T_{\lfloor (i+1) m \epell\rfloor, \lfloor (i+1) m \epell \rfloor + \lfloor i (n-m) \epell\rfloor}.$$ %$\text{supp}(F^\circ)\subseteq T_{\lfloor (i+1) m \epell \rfloor, \lfloor (i+1) m \epell \rfloor + \lfloor i(n-m) \epell  \rfloor}\setminus \mathcal{N}''$.
\end{definition}

\begin{example}
A matrix presentation of a polynomial in $\mathcal{R}$ may help understand divisibility conditions better. For $f\in\mathcal{R}$, let $m'=(0,1)\text{-deg}(f)$ and $n'=(1,0)\text{-deg}(f)$. Let $M(f)$ be the $m'\times n'$ matrix whose $(i,j)$-entry is equal to the coefficient for the term $x^{j-1}y^{m'-i+1}$ of $f$. For instance, if 
$$f=3x^6y + 2x^5y^2 + x^3y^4 + 2x^4y^2 + 3x^2y^4 + 3xy^4 + y^4 + xy^2 + y^2 +1,$$then 
 $$M(f)=\begin{pmatrix}
1 & 3 & 3 & 1 & 0 & 0 & 0\\
 0 & 0 & 0 & 0 & 0 & 0 & 0\\
 1 & 1 & 0 & 0 & 2 & 2 & 0\\
 0 & 0 & 0 & 0 & 0 & 0 & 3\\
 1 & 0 & 0 & 0 & 0 & 0 & 0\\
 \end{pmatrix}.$$ 
Using $M(f)$, it is easy to see that  $f^{(0,1)}_{j}$ is divisible by $(x+1)^{j-1}$  and $f^{(1,1)}_{j}$ is divisible by $(x+y)^{j-6}$ for each $j\in\mathbb{Z}$.
 
\end{example}

For any integers $r_1$ and $r_2$, the set $\{x\in \mathbb{Z} \ : \ r_1\le x\le r_2\}$ will be denoted by the usual notation $[r_1,r_2]$. 
 We shall show the following Conjecture~\ref{Jac_conj4} implies Conjecture~\ref{Jac_conj3} (see Lemma~\ref{DtoC}), hence implies the Jacobian conjecture.

\begin{conjecturealpha}\label{Jac_conj4}
Assume that $\ff$ and $\GG$ satisfy the conditions {\rm(1)--(5)} in Conjecture \ref{Jac_conj2}. Then the following statements hold true.

{\rm(i)} Let $w=(0,1)$ and $i\in [m(n-m)(a-1)/a+1, m(n-m)]$. If  $\ff^\circ$ satisfies the $i$-th collection of $w$-DC%and $w\text{-deg}(\ff^\circ)< \lfloor i (n-m) \epell\rfloor+m-m/a$
, then $\ff^\circ$ satisfies the $(i-1)$-th collection of $w$-DC.

{\rm(ii)} Let $w=(1,1)$ and $i\in [m(n-m)(a-1)/a, m(n-m)]$. If $\ff^\circ$ satisfies the $i$-th collection of $w$-DC%and $w\text{-deg}(\ff^\circ)< \lfloor (i+1) m \epell\rfloor+n-n/a$
, then $\ff^\circ$ satisfies the $(i-1)$-th collection of $w$-DC.

{\rm(iii)} Let $i\in [1,m(n-m)(a-1)/a]$. If $\ff^\circ$ satisfies the $i$-th collection of DSC, then $\ff^\circ$ satisfies the $(i-1)$-th collection of DC.
 \end{conjecturealpha}

%Note that obviously the $i$-th collection of divisibility condition is weaker than the $(i-1)$-th collection; in other words, the $i$-th collection of divisibility conditions becomes stronger as $i$ decreases. 

For convenience, denote 
\begin{equation}\label{aibi}
a_i=\lfloor (i+1) m \epell \rfloor,\quad b_i=\lfloor (i+1) m \epell \rfloor + \lfloor i(n-m) \epell \rfloor
\end{equation} 
Note that 
$(a_{m(n-m)},b_{m(n-m)})=(m, n),\; (a_0,b_0)=(0,0)$, and
\begin{equation}\label{xiyi}
 (a_i,b_i)-(a_{i-1},b_{i-1})\in\{(0,0),(0,1),(1,1)\}.
\end{equation}
Note that the above difference cannot be $(1,2)$ is because of the following simple lemma.
\begin{lemma}\label{not_both}
There does not exist $i$ such that $ i/m\in \mathbb{Z}\text{ and }(i+1)/(n-m)\in \mathbb{Z}$.
There does not exist $i$ such that $\lfloor i(n-m) \epell  \rfloor >\lfloor (i-1)(n-m) \epell  \rfloor$ and $\lfloor (i+1) m \epell \rfloor> \lfloor i m \epell \rfloor$.
\end{lemma}
\begin{proof}
If both strict inequalities hold, then $m|i$ and $(n-m)|(i+1)$. But this is impossible since both $m$ and $n-m$ are multiples of $a\ge2$.
\end{proof}

\begin{lemma}\label{multiple_adelta}
Fix $i\in [0,m(n-m)(a-1)/a]$. Assume that 
\begin{equation}\label{multiple_of}a\delta\lfloor (i+1) m \epell\rfloor/m = a\delta \lfloor i (n-m) \epell\rfloor/(n-m) \in\mathbb{Z}_{\ge0}.\end{equation}
If $\ff^\circ$ satisfies the $i$-th collection of DSC, then  $$\emph{supp}(F^\circ)\subseteq 
T_{a_i,b_i-1}.$$ 
%T_{\lfloor (i+1) m \epell\rfloor, \lfloor (i+1) m \epell\rfloor + \lfloor i (n-m) \epell\rfloor -1}.$$ 
 \end{lemma}
 \begin{proof}
Let integer $k$ be the value of \eqref{multiple_of}. Then  $\lfloor (i+1) m \epell \rfloor=\frac{mk}{a\delta}$ and $\lfloor i(n-m) \epell  \rfloor=\frac{(n-m)k}{a\delta}$.
The assumption that `` $\ff^\circ$ satisfies the $i$-th collection of DSC '' has the following two consequences:

\noindent (a) $\text{supp}(\ff^\circ)\subseteq T_{mk/a\delta,nk/a\delta}$. 

\noindent (b) 
 $(F^\circ)^{(0,1)}_{j}$ is divisible by $(x+1)^{j-\lfloor i(n-m) \epell  \rfloor}$ for $j=\lfloor (i+1) m \epell \rfloor + \lfloor i(n-m) \epell  \rfloor=nk/a\delta$, and
 $(F^\circ)^{(1,1)}_{j}$ is divisible by $(x+y)^{j-\lfloor (i+1) m \epell \rfloor}$ for $j=2\lfloor (i+1) m \epell \rfloor + \lfloor i(n-m) \epell  \rfloor=(m+n)k/a\delta$. That is,
$$(x+1)^{mk/a\delta}\ | \ (F^\circ)^{(0,1)}_{nk/a\delta}, \quad (x+y)^{nk/a\delta}\ | \ (F^\circ)^{(1,1)}_{(m+n)k/a\delta}$$ 
\smallskip

Let $\lambda\in\mathbb{C}$ be the coefficient of $x^{mk/a\delta}y^{nk/a\delta}$ in $F^\circ$. It follows from (a) and (b) that
$$(F^\circ)^{(0,1)}_{nk/a\delta}=\lambda(x+1)^{mk/a\delta}y^{nk/a\delta}\quad\text{and}\quad(F^\circ)^{(1,1)}_{(m+n)k/a\delta} = \lambda x^{mk/a\delta}(x+y)^{nk/a\delta}.$$ 
Therefore to prove the lemma, it suffices to prove that $\lambda=0$.

Assume the contrary that $\lambda\neq0$. Then $N^0(F^\circ)=T_{mk/a\delta,nk/a\delta}$. 
Let $\g'(z)=\g^\circ(z)+\lambda z^k$  and let $F'=F-\g'(E^\circ)=F-\g^\circ(E^\circ)-\lambda(E^\circ)^k=F^\circ-\lambda(E^\circ)^k$. Note that the assumption $i\le m(n-m)(a-1)/a$ implies 
$$k=a\delta\lfloor i/m \rfloor/(n-m)\le a\delta-\delta\le a\delta-1=\deg(\g^\circ)-1,$$ 
where the first ``$\le$'' becomes ``$=$'' when $i=m(n-m)(a-1)/a$, the second  ``$\le$'' becomes ``$=$'' when $\delta=1$.
So $k\le\deg(\g^\circ)-1$, and the ``$=$'' holds if and only if $i=m(n-m)(a-1)/a$ and $\delta=1$. We consider two cases:

\noindent Case 1: If $i=m(n-m)(a-1)/a$ and $\delta=1$. Then $k=a-1$, and $F^\circ$ contains a nonzero term $\lambda x^{m(a-1)/a}y^{n(a-1)/a}$. But then the point $(m(a-1)/a,n(a-1)/a)\in \text{supp}(F^\circ)$ is the southwest vertex of $\mathcal{N}''$, contradicting the assumption that $\text{supp}(F^\circ)\subseteq N(F)\setminus \mathcal{N}''$.

\noindent Case 2: If $i<m(n-m)(a-1)/a$ or $\delta>1$. Then $k\le \deg(g^\circ)-2$, thus $g'(z)\in \mathbb{T}_F$.
By Lemma~\ref{Ecirctopdegpieces}, the top (0,1)-homogeneous part of $\lambda(E^\circ)^k$ is $\lambda((x+1)^{m/a\delta}y^{n/a\delta})^k=\lambda(x+1)^{mk/a\delta}y^{nk/a\delta}$, which is exactly the top (0,1)-homogeneous part of $F^\circ$. Thus $(0,1)\text{-deg}(F')<nk/a\delta$. But this yields $N^0(F')\subsetneq N^0(F^\circ)\  \big(\subseteq N^0(\ff)\setminus \mathcal{N}''\big)$, which contradicts the assumption that $V^\circ$ is minimal as given in Definition \ref{def:VF}.
 \end{proof}

\section{Implication from Conjecture~\ref{Jac_conj4} to Conjecture~\ref{Jac_conj}}

Given a polynomial $F\in \mathbb{C}[x,y]$, define polynomials $F_{zy}, F_{xw}, F_{zw}$ by the condition $F_{zy}(z,y)=F_{xw}(x,w)=F_{zw}(z,w)=F(x,y)$. Equivalently, define the following isomorphisms of $K$-algebras:
$$\aligned
& \phi_{zy}\in\text{Hom}(K[x,y], K[z,y]): x\mapsto z-1, y\mapsto y, \\
&\phi_{zy}^{-1}\in\text{Hom}(K[z,y], K[x,y]): z\mapsto x+1, y\mapsto y, \\
&\phi_{xw}\in\text{Hom}(K[x,y], K[x,w]): : x\mapsto x, y\mapsto w-x, \\
&\phi_{xw}^{-1}\in\text{Hom}(K[x,w], K[x,y]): x\mapsto x, w\mapsto x+y,\\
&\phi_{zw}\in \text{Hom}(K[x,y], K[z,w]): x\mapsto z-1, y\mapsto w+1-z. \\
&\phi_{zw}^{-1}\in \text{Hom}(K[z,w],K[x,y]): z\mapsto x+1, w\mapsto x+y, \\
\endaligned
$$ 
Then $F_{zy}=\phi_{zy}(F)$, $F_{xw}=\phi_{xw}(F)$, $F_{zw}=\phi_{zw}(F)$. 

Let $T^1_{m,n}$ be the quadrilateral with vertices $(0,0),(m+n,0),(m,n),(n-m,0)$,
$T^2_{m,n}$ be the rectangle with vertices $(0,0),(m,0),(m,n),(n,0)$, and 
$T^3_{m,n}$ be the trapezoid with vertices $(0,0)$, $(m,0)$, $(m,n)$, $(n-m,0)$. 
(See Figure \ref{fig:NNN}.)

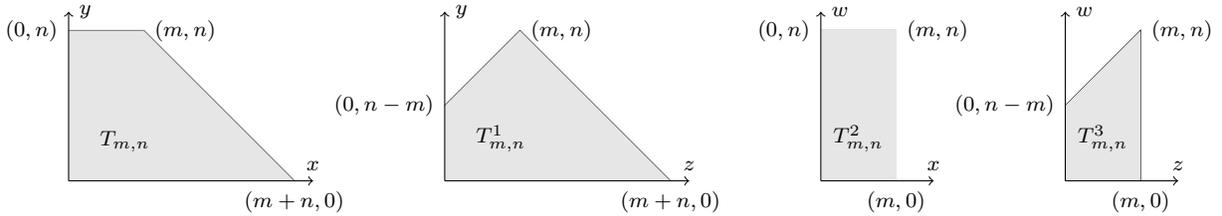
\begin{figure}[h]
\begin{center}
\begin{tikzpicture}[scale=0.25]
%1
\draw (0,8)--(4,8)--(12,0);
\fill[black!10] (0,0)--(0,8)--(4,8)--(12,0)--(0,0);
\draw (4, 8) node[anchor=west] {\tiny $(m,n)$};
\draw (12, 0) node[anchor=north] {\tiny $(m+n,0)$};
\draw (0, 8) node[anchor=east] {\tiny $(0,n)$};
\draw (3, 1) node[anchor=south] {\tiny $T_{m,n}$};
\draw[->] (0,0) -- (13,0)
node[above] {\tiny $x$};
\draw[->] (0,0) -- (0,9)
node[right] {\tiny $y$};
%2
\begin{scope}[shift={(20,0)}]
\draw (0,4)--(4,8)--(12,0);
\fill[black!10] (0,0)--(0,4)--(4,8)--(12,0)--(0,0);
\draw (4, 8) node[anchor=west] {\tiny $(m,n)$};
\draw (12, 0) node[anchor=north] {\tiny $(m+n,0)$};
\draw (0, 4) node[anchor=east] {\tiny $(0,n-m)$};
\draw (3, 1) node[anchor=south] {\tiny $T^1_{m,n}$};
\draw[->] (0,0) -- (13,0)
node[above] {\tiny $z$};
\draw[->] (0,0) -- (0,9)
node[right] {\tiny $y$};
\end{scope}
%3
\begin{scope}[shift={(40,0)}]
\draw (0,8)--(4,8)--(4,0);
\fill[black!10] (0,0)--(0,8)--(4,8)--(4,0)--(0,0);
\draw (4, 8) node[anchor=west] {\tiny $(m,n)$};
\draw (4, 0) node[anchor=north] {\tiny $(m,0)$};
\draw (0, 8) node[anchor=east] {\tiny $(0,n)$};
\draw (2, 1) node[anchor=south] {\tiny $T^2_{m,n}$};
\draw[->] (0,0) -- (6,0)
node[above] {\tiny $x$};
\draw[->] (0,0) -- (0,9)
node[right] {\tiny $w$};
\end{scope}
%4
\begin{scope}[shift={(53,0)}]
\draw (0,4)--(4,8)--(4,0);
\fill[black!10] (0,0)--(0,4)--(4,8)--(4,0)--(0,0);
\draw (4, 8) node[anchor=west] {\tiny $(m,n)$};
\draw (4, 0) node[anchor=north] {\tiny $(m,0)$};
\draw (0, 4) node[anchor=east] {\tiny $(0,n-m)$};
\draw (2, 1) node[anchor=south] {\tiny $T^3_{m,n}$};
\draw[->] (0,0) -- (6,0)
node[above] {\tiny $z$};
\draw[->] (0,0) -- (0,9)
node[right] {\tiny $w$};
\end{scope}
\end{tikzpicture}
\end{center}
\caption{From left to right: $T_{m,n}, T^1_{m,n},T^2_{m,n}, T^3_{m,n}$. (Assume $n>m>0$.)}
\label{fig:NNN}
\end{figure}

\begin{lemma}\label{lemma:N(F) change variables}
Assume  $m<n$. Consider the following two conditions.

{\rm(a)} $F^{(0,1)}_{n-i}$ is divisible by $(x+1)^{m-i}$ for all $0\le i\le m$.

{\rm(b)} $F^{(1,1)}_{m+n-i}$ is divisible by $(x+y)^{n-i}$ for all $0\le i\le n$. 

 Then the following statements hold true.

{\rm(i)}  ``$N^0(\ff)\subseteq T_{m,n}$ and the condition {\rm(a)} holds'' if and only if $N^0(F_{zy})\subseteq T^1_{m,n}$. 

{\rm(ii)} ``$N^0(\ff)\subseteq T_{m,n}$ and the condition {\rm(b)} holds'' if and only if $N^0(F_{xw})\subseteq T^2_{m,n}$.

{\rm(iii)} ``$N^0(\ff)\subseteq T_{m,n}$ and both {\rm(a)} and {\rm(b)} hold'' if and only if  $N^0(F_{zw})\subseteq T^3_{m,n}$.
\end{lemma}

\begin{proof}
(i) If $N^0(\ff)\subseteq T_{m,n}$ and (a) holds, then $z^{m-i}$ divides $F^{(0,1)}_{n-i}$ for all $i$; that is, the support of $F_{zy}$ does not contain points $(j,n-i)$ for $j<m-i$; thus the support of $F_{zy}$ lies in the half plane $z\ge y-n+m$. Meanwhile,  since $\phi_{zy}$ sends $x\mapsto z-1$ and $y\mapsto y$, the polynomial $F_{zy}=\phi_{zy}(F)$ has the  same right boundary of the support as $F$. Thus the support of $F_{zy}$ lies in the intersection of $T_{m,n}$ and the halfplane $z\ge y-n+m$, which is $T^1_{m,n}$. The converse is obvious. 

(ii) is similar to (i). 
If $N^0(\ff)\subseteq T_{m,n}$ and (b) holds, then $w^{n-i}$ divides $F^{(1,1)}_{m+n-i}$ for all $i$; that is, the support of $F_{xw}$ does not contain points $(m+n-i-j,j)$ for $j<n-i$; thus the support of $F_{xw}$ lies in the half plane $x\le m$. Meanwhile, $\phi_{xw}$ sends $x\mapsto x$ and $y\mapsto w-x$, so the polynomial $F_{xw}=\phi_{xw}(F)$ and $F$ have the same top horizontal boundary of the support. Thus the support of $F_{xw}$ lies in the intersection of $T_{m,n}$ and the halfplane $x\le m$, which is $T^2_{m,n}$. The converse is obvious.

(iii) First we prove the ``only if'' part. The polynomial $F_{zw}$ can be obtained from $F_{zy}$ by $y\mapsto w+1-z$. For $(i,j)\in T^1_{m,n}$, since $z^iy^j=z^i(w+1-z)^j$ whose support is in the right triangle with vertices $(i,0),(i,j),(i+j,0)$ (which is contained in $T^1_{m,n}$), we see that $N^0(F_{zw})\subseteq T^1_{m,n}$.
Similarly, $F_{zw}$ can be obtained from $F_{xw}$ by $x\mapsto z-1$. For $(i,j)\in T^2_{m,n}$, since $x^iw^j=(z-1)^iw^j$ whose support is in the horizontal line segment with endpoints $(0,j),(i,j)$ (which is contained in $T^2_{m,n}$), we see that $N^0(F_{zw})\subseteq T^2_{m,n}$.
Thus $N^0(F_{zw})\subseteq T^1_{m,n}\cap T^2_{m,n}=T^3_{m,n}$. 

Next we prove the ``if'' part:  $N^0(\ff)\subseteq T_{m,n}$ is obvious. Meanwhile,
$N^0(F_{zw})\subseteq T^3_{m,n} \Rightarrow N^0(F_{zy})\subseteq T^1_{m,n}\Rightarrow (a)$;   
$N^0(F_{zw})\subseteq T^3_{m,n} \Rightarrow N^0(F_{xw})\subseteq T^2_{m,n}\Rightarrow (b)$;
so both (a) and (b) hold.
\end{proof}

\begin{lemma}\label{BtoA}
Conjecture~\ref{Jac_conj2} implies Conjecture~\ref{Jac_conj}.
\end{lemma}
\begin{proof}
We first recall a known result (see \cite{Na1} or \cite{ML1} for a proof): if $f,g\in\mathbb{C}[x,y]$ satisfy $[f,g]=1$, then there exists an automorphism $\xi$ of $\mathbb{C}[x,y]$ such that $N^0(\xi(f))$ is contained in the trapezoid $T^3_{m,n}$ and contains the vertex $(m,n)$ where $n>m>0$. \footnote{Indeed, a stronger result in \cite{CN} or \cite{GGV} asserts that we can further assume that  $(m,n)$ is the only point in the support of $\xi(f)$ that lies on the top boundary  of the trapezoid $T^3_{m,n}$ (that is, the boundary line segment of $T^3_{m,n}$ with slope $1$); we do not need this stronger result.}

Assuming Conjecture~\ref{Jac_conj2} is true. To prove Conjecture~\ref{Jac_conj}, it is enough to assume that $f,g\in \mathbb{C}[x,y]$ satisfy the conditions $[f,g]\in \mathbb{C}$, and $\{(m,n)\}\subseteq N^0(f)\subseteq T^3_{m,n}$ for some $n>m>0$, and then prove $[f,g]=0$. 

Multiplying $f,g$ with nonzero constants if necessary, We can assume that $N^0(f-x^my^n)\subsetneq N^0(f)$ and $N^0(g-x^{bm/a}y^{bn/a})\subsetneq N^0(g)$.
Let $F_{zw}\in \mathbb{C}[z,w]$ (resp.~$G_{zw}$) be obtained from $f$ (resp.~$g$) by simply replacing the variables $x$ by $z$ and $y$ by $w$. Define $F,G\in \mathbb{C}[x,y]$ by 
$$F=\phi_{zw}^{-1}(F_{zw}),\quad G=\phi_{zw}^{-1}(G_{zw}).$$

We claim that the conditions (1)--(5) in Conjecture~\ref{Jac_conj2}  are satisfied:

(1)(4)  Since $N^0(F_{zw})\subseteq T^3_{m,n}$,  we have $N^0(F)\subseteq T_{m,n}$ by Lemma \ref{lemma:N(F) change variables}. Since $f$ has a term $1x^my^n$, $F_{zw}$ has a term $1z^mw^n$, thus $F$ has terms $1x^my^n$, $1x^{m+n}$ and $1y^n$. It follows that that  $N^0(\ff - x^m y^n)\subsetneq N^0(\ff)$ and $N^0(F)=T_{m,n}$.  
To see $N^0(\GG)=T_{bm/a,bn/a}$, note that $[F,G]\in\mathbb{C}$ implies $N^0(G)$ is similar to $N^0(F)$.
The condition $N^0(g-x^{bm/a}y^{bn/a})\subsetneq N^0(g)$ implies $N^0(\GG - x^{bm/a} y^{bn/a})\subsetneq N^0(\GG)$. 
    
(2)(3) By Lemma \ref{lemma:N(F) change variables}, since $N^0(F_{zw})\subseteq T^3_{m,n}$, both conditions (a) and (b) in Lemma \ref{lemma:N(F) change variables} must hold. These two conditions are exactly (2) and (3).

(5) Since $[f, g]\in \mathbb{C}$, we have $[f,g]=[F_{zw},G_{zw}]=[F,G]\in \mathbb{C}$, where the second equality is by the chain rule.
%Let $u=\phi(x), v=\phi(y)$. As $\phi\in\text{Aut}(K[x,y])$, $\phi(f)=f(u,v)$ and $ \phi(g)= g(u,v)$. Note $\frac{\partial u}{\partial x}=\frac{\partial v}{\partial x}=\frac{\partial v}{\partial u}=1$ and $\frac{\partial u}{\partial y}=0$. Since $[f, g]\in K$, by the chain rule, $$[\phi(f), \phi(g)]=[f(u,v), g(u,v)] = \det \begin{pmatrix}{\partial f}/{\partial u}+{\partial f}/{\partial v} & {\partial g}/{\partial u}+{\partial g}/{\partial v}\\ {\partial f}/{\partial v} &{\partial g}/{\partial v}\end{pmatrix}$$ $$=\det \begin{pmatrix}{\partial f}/{\partial u} & {\partial g}/{\partial u}\\ {\partial f}/{\partial v} &{\partial g}/{\partial v}\end{pmatrix} = [f, g]\in K $$    

Conjecture~\ref{Jac_conj2} then asserts that $[f,g]=[F,G]=0$, therefore it implies  Conjecture~\ref{Jac_conj}.
\end{proof}

\begin{lemma}\label{CtoB}
Conjecture~\ref{Jac_conj3} implies Conjecture~\ref{Jac_conj2}.
\end{lemma}
\begin{proof}
It is clear that (i)--(iv) are equivalent. Assuming Conjecture~\ref{Jac_conj3} is true, it suffices to prove the following: if there is $\g\in \mathbb{T}$ such that $\ff - \g(E^\circ)=0$, then $[F,G]=0$.

By Lemma \ref{Ecirctopdegpieces}(ii),  $\deg(\alpha)=a\delta\ge a>1$. 
Using $F=\alpha(E^\circ)$, we get
$$
[F, G]  = \frac{d\g}{dz}(E^\circ) [E^\circ, G] \in K.
$$
Since $\deg\g>1$, $\frac{d\g}{dz}(E^\circ)$ is a nonconstant polynomial, consequently $[E^\circ,G]=0$, thus $[\ff, G]=0$. 
\end{proof}

\begin{proposition}\label{PQdivisible}
Let $\ff$ satisfy the conditions {\rm(1)--(3)} in Conjecture \ref{Jac_conj2}.  Let $Q$ be the unique polynomial obtained from Lemma \ref{lem:c=p^2 most generalized}, and let $R=F-Q^a$. Then the following holds. 

{\rm(i)} $Q^{(0,1)}_{\frac{n}{a}-i}$ is divisible by $(x+1)^{\frac{m}{a}-i}$ for all $0\le i\le \frac{m}{a}$, and all monomials of $R$ have (0,1)-degree $\le n-\frac{m}{a}-1$.% that is, $\supp(R)\subset (N(F)\setminus \mathcal{N}'' )\cap\mathbb{Z}\times[1,n-\frac{m}{a}-1]$. 
Consequently, $(E^\circ)^{(0,1)}_{\frac{n}{a\delta}-i}$ is divisible by $(x+1)^{\frac{m}{a\delta}-i}$ for all $0\le i\le \frac{m}{a\delta}$. 

{\rm(ii)} Assume that all monomials of $R$ have (0,1)-degree $\le n-\frac{n-m}{a}$. 
Then  $Q^{(1,1)}_{\frac{m+n}{a}-i}$ is divisible by $(x+y)^{\frac{n}{a}-i}$ for all $0\le i\le\frac{n}{a}$, and all monomials of $R$ have (1,1)-degree $<m+n-\frac{n}{a}$. 
Consequently, $(E^\circ)^{(1,1)}_{\frac{m+n}{a\delta}-i}$ is divisible by $(x+y)^{\frac{n}{a\delta}-i}$ for all $0\le i\le \frac{n}{a\delta}$. 
\end{proposition}
%{\color{red} [Note the statement on $E^\circ$ is used in Lemma 5.2, so cannot be deleted.]}

\begin{proof}
\noindent (i)
Lemma~\ref{lem:c=p^2 most generalized} implies that $R$ contains no monomial of the form $x^{i}y^{j}$ with $i\ge m(a-1)/a$ and $j\ge n(a-1)/a$. Define $R_{zy}=\phi_{zy}(R)=R(z-1,y)$. 

Note that $R_{zy}$ must contain no $z^{\ge m(a-1)/a}y^{\ge n(a-1)/a}$. To see this, assume the contrary that $R_{zy}$ contains a monomial $z^{m'}y^{n'}$ where $m'\ge m(a-1)/a$ is maximal for fixed $n'\ge n(a-1)/a$. Then $R=R_{zy}(x+1,y)$ contains a monomial $x^{m'}y^{n'}$ that will not cancel out with another term in $R$. This contradicts the assumption on $R$.

Now we can prove by induction on $i$. 

For $i=0$, $Q^{(0,1)}_{n/a}$ is uniquely determined by the conditions that it contains a term $1x^{m/a}y^{n/a}$ and that $(Q^{(0,1)}_{n/a})^a-F^{(0,1)}_n$ does not contain monomials of the form $x^{\ge m(a-1)/a} y^n$. 
Since $F^{(0,1)}_n=(x+1)^my^n$, we must must have $Q^{(0,1)}_{n/a} =  (x+1)^{m/a}y^{n/a}$ by uniqueness (because it satisfies the above conditions).  So the statement is true for the base case.

For $0<i\le \frac{m}{a}$, $Q^{(0,1)}_{\frac{n}{a}-i}$ is uniquely determined by the condition that the following polynomial contains  no monomials of the form $x^{\ge m(a-1)/a}y^{\ge n(a-1)/a}$:
\begin{equation}\label{eq:pre E-}
R^{(0,1)}_{n-i}=F^{(0,1)}_{n-i}-\sum_{k_1+k_2+\cdots+k_a=i} Q^{(0,1)}_{\frac{n}{a}-k_1}Q^{(0,1)}_{\frac{n}{a}-k_2}\cdots Q^{(0,1)}_{\frac{n}{a}-k_a}%\in\mathbb{C}[x,y]. 
\end{equation}

We single out the case when $k_b=i$ for some $b$,  and $k_j=0$ for other $j\neq b$. The sum over all such $(k_1,\dots,k_a)$ is equal to 
$a(Q^{(0,1)}_{\frac{n}{a}})^{a-1}Q^{(0,1)}_{\frac{n}{a}-i}= a(x+1)^{\frac{m(a-1)}{a}}y^{\frac{n(a-1)}{a}}Q^{(0,1)}_{\frac{n}{a}-i}$. 

Applying $\phi_{zy}$ to \eqref{eq:pre E-}, the condition becomes that, the following polynomial in $\mathbb{C}[z,y]$,
\begin{equation}\label{eq:E-}
\phi_{zy}(R^{(0,1)}_{n-i})=\phi_{zy}(F^{(0,1)}_{n-i})-az^{\frac{m(a-1)}{a}}y^{\frac{n(a-1)}{a}}\phi_{zy}(Q^{(0,1)}_{\frac{n}{a}-i})-\sum_{k_1,\dots,k_a} \phi_{zy}(Q^{(0,1)}_{\frac{n}{a}-k_1}Q^{(0,1)}_{\frac{n}{a}-k_2}\cdots Q^{(0,1)}_{\frac{n}{a}-k_a}),
\end{equation}
contains no monomials of the form $z^{\ge \frac{m(a-1)}{a}}y^{\ge \frac{n(a-1)}{a}}$ (where the sum runs through all tuples $(k_1,\dots,k_a)$ of nonnegative integers less than $i$ and satisfying $k_1+\cdots+k_a=i$).
 
Note that $z^{m-i}|\phi_{zy}(F^{(0,1)}_{n-i})$ by assumption.  Also note that  $z^{\frac{m}{a}-k_j}|\phi_{zy}(Q^{(0,1)}_{\frac{n}{a}-j})$ (for $j=1,\dots,a$) by inductive hypothesis, thus $z^{m-i}|\phi_{zy}(Q^{(0,1)}_{\frac{n}{a}-k_1}Q^{(0,1)}_{\frac{n}{a}-k_2}\cdots Q^{(0,1)}_{\frac{n}{a}-k_a})$. Since $m(a-1)/a \le m-i$, all terms in \eqref{eq:E-} are divisible by $z^{\frac{m(a-1)}{a}}y^{n-i}$, where the exponent of $y$ satisfies  $n-i\ge n-\frac{m}{a}>n-\frac{n}{a}=\frac{n(a-1)}{a}$.
But \eqref{eq:E-} contains no monomials of the form $z^{\ge \frac{m(a-1)}{a}}y^{\ge \frac{n(a-1)}{a}}$. This forces \eqref{eq:E-} to be 0. This has two consequences:

First,  the second term in \eqref{eq:E-}, which is 
$az^{\frac{m(a-1)}{a}}y^{\frac{n(a-1)}{a}}\phi_{zy}(Q^{(0,1)}_{\frac{n}{a}-i})$, 
must be divisible by $z^{m-i}$ since all other terms are divisible by $z^{m-i}$. So $z^{\frac{m}{a}-i}|\phi_{zy}(Q^{(0,1)}_{\frac{n}{a}-i})$, thus $(x+1)^{\frac{m}{a}-i}|Q^{(0,1)}_{\frac{n}{a}-i}$.

Second, $\phi_{zy}(R^{(0,1)}_{n-i})=0$, thus $R^{(0,1)}_{n-i}=0$,  for $0\le i\le \frac{m}{2}$. For $i=\frac{m}{2}$, we conclude that all monomials of $R$ has $(0,1)$-degree $\le n-\frac{m}{2}-1$.

Third, we prove that $(E^\circ)^{(0,1)}_{\frac{n}{a\delta}-i}$ is divisible by $(x+1)^{\frac{m}{a\delta}-i}$ for all $i$. 
Since $N^0(Q)\subseteq\mathcal{N}'=\frac{1}{a}N^0(F)=T_{\frac{m}{a},\frac{n}{a}}$, we have $N^0(\phi_{zy}(Q))\subseteq T^1_{\frac{m}{a},\frac{n}{a}}$ by Lemma \ref{lemma:N(F) change variables}. 
Then by \eqref{eq:QE}, we have 
\begin{equation}\label{eq:N0EinT}
N^0(\phi_{zy}(E^\circ))=\frac{1}{\delta}N^0(\phi_{zy}(Q))\subseteq T^1_{\frac{m}{a\delta},\frac{n}{a\delta}}
\end{equation}
Applying Lemma \ref{lemma:N(F) change variables} again, we conclude that $(E^\circ)^{(0,1)}_{\frac{n}{a\delta}-i}$ is divisible by $(x+1)^{\frac{m}{a\delta}-i}$ for all $i$.

This finishes the proof of (i).

\smallskip

(ii) will be proved similarly. We prove by induction on $i$. 

For $i=0$, $Q^{(1,1)}_{\frac{m+n}{a}}=  x^{\frac{m}{a}}(x+y)^{\frac{n}{a}}$, so the statement is true in this case.

For $0<i\le \frac{n}{a}$, $Q^{(1,1)}_{\frac{m+n}{a}-i}$ is uniquely determined by the condition that the following polynomial contains no $x^{\ge \frac{m(a-1)}{a}}y^{\ge \frac{n(a-1)}{a}}$:
\begin{equation}\label{eq:F-}
R^{(1,1)}_{m+n-i}=F^{(1,1)}_{m+n-i}-ax^{\frac{m(a-1)}{a}}(x+y)^{\frac{n(a-1)}{a}}Q^{(1,1)}_{\frac{m+n}{a}-i}-\sum_{k_1,\dots,k_a} 
Q^{(1,1)}_{\frac{m+n}{a}-k_1}Q^{(1,1)}_{\frac{m+n}{a}-k_2}\cdots Q^{(1,1)}_{\frac{m+n}{a}-k_a}. \end{equation}

We claim that \eqref{eq:F-} must be 0. 

If not, suppose it contains a monomial $x^ky^{m+n-i-k}$ for some $k$, and we let $k$ be the smallest choice. Then either $k\le m(a-1)/a-1$, or $m+n-i-k\le n(a-1)/a-1$. We consider these two cases:

Case 1. If $k\le m(a-1)/a-1$, then the exponent of $y$ is $m+n-i-k\ge m+n-\frac{n}{a}-\frac{m(a-1)}{a}+1=n-\frac{n-m}{a}+1$. Since we assumed that all monomials of $R$ have $(0,1)$-degree $\le n-\frac{n-m}{a}$,
% in other words, the exponent of $y$ must be $\le n-\frac{m}{2}$. 
\eqref{eq:F-} must be 0.

Case 2. If $m+n-i-k\le n(a-1)/a-1$, then the $(0,1)$-degree of $\eqref{eq:F-}$ is $\le n(a-1)/a-1$. 
Applying $\phi_{xw}$ to \eqref{eq:F-}, we get
\begin{align}\label{eq:F-after}
\phi_{xw}(R^{(1,1)}_{m+n-i})=\phi_{xw}(F^{(1,1)}_{m+n-i})-ax^{\frac{m(a-1)}{a}}w^{\frac{n(a-1)}{a}}\phi_{xw}(Q^{(1,1)}_{\frac{m+n}{a}-i})
\\
-\sum_{k_1,\dots,k_a} 
\phi_{xw}(Q^{(1,1)}_{\frac{m+n}{a}-k_1}Q^{(1,1)}_{\frac{m+n}{a}-k_2}\cdots Q^{(1,1)}_{\frac{m+n}{a}-k_a}). \notag%
\end{align}
and that the degree of $w$ of each monomial of $\phi_{xw}(R^{(1,1)}_{m+n-i})$ should be $\le n(a-1)/a-1$. However, $w^{n(a-1)/a}$ divides every term in  \eqref{eq:F-after}: 

for the first term: $w^{n-i}|\phi_{xw}(F^{(1,1)}_{m+n-i})$ and $n-i\ge n-\frac{n}{a}=\frac{n(a-1)}{a}$ imply $w^{\frac{n(a-1)}{a}} |\phi_{xw}(F^{(1,1)}_{m+n-i})$; 

for the second term: it is obviously divisible by $w^{\frac{n(a-1)}{a}}$; 

for each product in the sum: since $w^{\frac{n}{a}-k_j}|\phi_{xw}(Q^{(1,1)}_{\frac{m+n}{a}-k_j})$ for each $j$, we conclude that $\phi_{xw}(\prod_{j=1}^a Q^{(1,1)}_{\frac{m+n}{a}-k_j})$ is divisible by $w^{n-i}$, therefore is divisible by $w^{\frac{n(a-1)}{a}}$ since $n-i\ge \frac{n(a-1)}{a}$. 

This contradiction forces \eqref{eq:F-after} and  \eqref{eq:F-} to be 0. 

Then we complete the proof similarly to the proof of (i): all terms in  \eqref{eq:F-after} are divisible by $w^{n-i}$, by studying the second term of \eqref{eq:F-after} we conclude $w^{\frac{n}{a}-i}|\phi_{xw}(Q^{(1,1)}_{\frac{m+n}{a}-i})$; the (1,1)-homogeneous part of $R$ of degree $m+n-i$ is 0 for $0\le i\le \frac{n}{a}$, so all monomials of $R$ has $(1,1)$-degree $\le m+n-\frac{n}{a}-1$. 

Last, we prove that 
$(E^\circ)^{(1,1)}_{\frac{m+n}{a\delta}-i}$ is divisible by $(x+y)^{\frac{n}{a\delta}-i}$ for all $i$. 
Since $N^0(Q)\subseteq\mathcal{N}'=\frac{1}{a}N^0(F)=T_{\frac{m}{a},\frac{n}{a}}$, we have $N^0(Q_{xw})\subseteq T^2_{\frac{m}{a},\frac{n}{a}}$ by Lemma \ref{lemma:N(F) change variables}. 
Then by \eqref{eq:QE}, we have $N^0(E^\circ_{xw})=\frac{1}{\delta}N^0(Q_{xw})\subseteq T^2_{\frac{m}{a\delta},\frac{n}{a\delta}}$. Applying Lemma \ref{lemma:N(F) change variables} again, we conclude that $(E^\circ)^{(1,1)}_{\frac{m+n}{a\delta}-i}$ is divisible by $(x+y)^{\frac{n}{a\delta}-i}$ for all $i$. 
\end{proof}

%\begin{remark}
%Proposition~\ref{PQdivisible} can be proved also by using the arguments in \cite{GGV14}. See Appendix. 
%\end{remark}

\begin{lemma}\label{Div_to_limit1}
 Let $w=(0,1)$ and $i\in [m(n-m)(a-1)/a, m(n-m)]$. If $\ff^\circ$ satisfies the $i$-th collection of $w$-DC, then $w\text{-deg}(\ff^\circ)< \lfloor i(n-m) \epell  \rfloor+m-m/a$.
\end{lemma}
\begin{proof}
Let $j$ be an integer with $j\ge\lfloor i(n-m) \epell  \rfloor+m-m/a\; (=b_i-a_i+m\frac{a-1}{a})$. Then $j\ge \lfloor m(n-m)\frac{a-1}{a}\cdot(n-m)\epell\rfloor+m\frac{a-1}{a}%=(n-m)\frac{a-1}{a}+m\frac{a-1}{a}
=n\frac{a-1}{a}$ (the $y$-coordinate of the bottom edge of the trapezoid $\mathcal{N}''$). So $(1,0)\text{-deg}(F^\circ)_j^{(0,1)}< m\frac{a-1}{a}$ (the $x$-coordinate of the left edge of $\mathcal{N}''$). 

If $\ff^\circ$ satisfies the $i$-th collection of $w$-DC, then $(x+1)^{j-(b_i-a_i)}|(F^\circ)_j^{(0,1)}$. If $(F^\circ)_j^{(0,1)}\neq0$, then  $(1,0)\text{-deg}(F^\circ)_j^{(0,1)}\ge {j-(b_i-a_i)}\ge m\frac{a-1}{a}$, contradicting the previous paragraph. So $(F^\circ)_j^{(0,1)}=0$ for all $j\ge\lfloor i(n-m) \epell  \rfloor+m-m/a$. 
\end{proof}

\begin{lemma}\label{Div_to_limit2}
Assume $(0,1)\text{-deg}(\ff^\circ)<n\frac{a-1}{a}$. Let $i\in [m(n-m)\frac{a-1}{a}-1, m(n-m)]$ and $w=(1,1)$. If $\ff^\circ$ satisfies the $i$-th collection of $w$-DC, then $w\text{-deg}(\ff^\circ)< \lfloor (i+1) m \epell \rfloor+n\frac{a-1}{a}$.
\end{lemma}
\begin{proof}
Let $j$ be an integer with $j\ge \lfloor (i+1) m \epell \rfloor+n-n/a\; (=a_i+n\frac{a-1}{a})$. Then $j\ge \lfloor (m(n-m)\frac{a-1}{a}-1+1)\cdot m \epell\rfloor+n\frac{a-1}{a}
=(m+n)\frac{a-1}{a}$.% (the $y$-coordinate of the bottom edge of the trapezoid $\mathcal{N}''$). So $(1,0)\text{-deg}(F^\circ)_j^{(0,1)}< m\frac{a-1}{a}$ (the $x$-coordinate of the left edge of $\mathcal{N}''$). 

If $\ff^\circ$ satisfies the $i$-th collection of $w$-DC, then $(x+y)^{j-a_i}|(F^\circ)_j^{(1,1)}$. If $(F^\circ)_j^{(1,1)}\neq0$, then  $(0,1)\text{-deg}(F^\circ)_j^{(1,1)}\ge {j-a_i}\ge n\frac{a-1}{a}$, contradicting the assumption. So $(F^\circ)_j^{(1,1)}=0$ for all $j\ge\lfloor (i+1)i m \epell \rfloor+n\frac{a-1}{a}$. 
\end{proof}

\begin{lemma}\label{DtoC}
Conjecture~\ref{Jac_conj4} implies Conjecture~\ref{Jac_conj3}.
\end{lemma}
\begin{proof}
Let $\ff$ and $\GG$ satisfy the conditions (1)--(5) in Conjecture \ref{Jac_conj2}. We want to show that Conjecture~\ref{Jac_conj4} implies $F^\circ=0$. Assume Conjecture~\ref{Jac_conj4} is true. 

(I) If $w=(0,1)$ and $i=m(n-m)$, then $\ff^\circ$ satisfies the $i$-th collection of $w$-DC.% and $w\text{-deg}(\ff^\circ)< \lfloor i(n-m) \epell  \rfloor+m-m/a$. 

Proof of (I):  in \eqref{eq:N0EinT} of the proof of Proposition~\ref{PQdivisible}(i), we see that $N^0(\phi_{zy}(E^\circ))\subseteq T^1_{\frac{m}{a\delta},\frac{n}{a\delta}}$. By Lemma \ref{Ecirctopdegpieces}(ii), $\deg(\alpha^\circ)=a\delta$, thus
$N^0(\phi_{zy}(\g^\circ(E^\circ)))=N^0(\g^\circ(\phi_{zy}(E^\circ)))= a\delta T^1_{\frac{m}{a\delta},\frac{n}{a\delta}}=T^1_{m,n}$, and 
$N^0(\g^\circ(E^\circ))= T_{m,n}$. Then by Lemma \ref{lemma:N(F) change variables} (i),
 $\g^\circ(E^\circ)$ satisfies the condition (a) of Lemma \ref{lemma:N(F) change variables}. Then both $F$ and  $\g^\circ(E^\circ)$, therefore $F^\circ$, satisfy the $i$-th collection of $w$-DC for $i=m(n-m)$. 

%Then, consider $(0,1)\text{-deg}(F^\circ)$. For any $j$ with $n(a-1)/a\le j\le n$, because of $\text{supp}(F^\circ)\subseteq N(\ff)\setminus\mathcal{N}''$, the support of $(F^\circ)^{(0,1)}_j$ is to the left side of $\mathcal{N}''$, thus the $(1,0)$-degree of $(F^\circ)^{(0,1)}_j$ is less than $m(a-1)/a$. However, $(F^\circ)^{(0,1)}_j$ is divisible by $(x+1)^{j-(n-m)}$ because of the $i$-th collection of $w$-DC. So if $(F^\circ)^{(0,1)}_j\neq0$, then $j-(n-m)<m(a-1)/a$, that is, $j<n-m/a$. 
%Therefore $(0,1)\text{-deg}(F^\circ)<n-m/a$. The proof of (I) is finished.
\smallskip

(II) $F^\circ$ satisfies the $m(n-m)\frac{a-1}{a}$-th collection of $(0,1)$-DC, and 
$(0,1)\text{-deg}(\ff^\circ)<n\frac{a-1}{a}$.

Proof of (II): we first consider the following modified version of Conjecture~\ref{Jac_conj4}(i):

{\rm(i$'$)} Let $w$ and $i$ be as in Conjecture~\ref{Jac_conj4}(i). If  $\ff^\circ$ satisfies the $i$-th collection of $w$-DC and $w\text{-deg}(\ff^\circ)< \lfloor i(n-m) \epell  \rfloor+m-m/a$, then $\ff^\circ$ satisfies the $(i-1)$-th collection of $w$-DC and $w\text{-deg}(\ff^\circ)< \lfloor (i-1)(n-m) \epell  \rfloor+m-m/a$.

Indeed, (i$'$) is equivalent to (i), thanks to Lemma~\ref{Div_to_limit1}. %the following claim (which shows that the condition on $w\text{-deg}(\ff^\circ)$ is redundant):

Since we assume that Conjecture~\ref{Jac_conj4}(i) is true, (i$'$) holds true, too. Then by a downward induction on $i$ with the base case proved in (I), we 
conclude that  $F^\circ$ satisfies the $m(n-m)\frac{a-1}{a}$-th collection of $w$-DC. Then Lemma~\ref{Div_to_limit1} implies 
$w\text{-deg}(\ff^\circ)< \lfloor m(n-m)\frac{a-1}{a}\cdot(n-m) \epell  \rfloor+m\frac{a-1}{a}=n\frac{a-1}{a}$. So (II) is proved.
\smallskip

(III) If $w=(1,1)$ and $i=m(n-m)$, then $\ff^\circ$ satisfies the $i$-th collection of $w$-DC.% and $w\text{-deg}(\ff^\circ)< \lfloor i(n-m) \epell  \rfloor+m-m/a$. 

\noindent The proof of (III) is similar to the proof of (I), so we skip.
\smallskip

(IV) If $w=(1,1)$, then $F^\circ$ satisfies the $m(n-m)\frac{a-1}{a}$-th collection of $w$-DC, and 
$w\text{-deg}(\ff^\circ)<(m+n)\frac{a-1}{a}.$

Proof of (IV): it is similar to the proof of (II). We consider the following modified version of Conjecture~\ref{Jac_conj4}(ii):

{\rm(ii$'$)} Let $w$ and $i$ be as in Conjecture~\ref{Jac_conj4}(ii). If $\ff^\circ$ satisfies the $i$-th collection of $w$-DC and $w\text{-deg}(\ff^\circ)< \lfloor (i+1) m \epell \rfloor+n-n/a$, then $\ff^\circ$ satisfies the $(i-1)$-th collection of $w$-DC and $w\text{-deg}(\ff^\circ)< \lfloor i m \epell \rfloor+n-n/a$.

Indeed, (ii$'$) is equivalent to (ii), thanks to Lemma~\ref{Div_to_limit2}.  %Indeed, (ii$'$) is equivalent to (ii), thanks to the following claim (which shows that the condition on $w\text{-deg}(\ff^\circ)$ is redundant):

Since we assume that Conjecture~\ref{Jac_conj4}(ii) is true, (ii$'$) holds true, too. Then by a downward induction on $i$ with the base case proved in (III), we 
conclude that  $F^\circ$ satisfies the $m(n-m)\frac{a-1}{a}$-th collection of $(1,1)$-DC. Then Lemma~\ref{Div_to_limit2} implies that 
$(1,1)\text{-deg}(\ff^\circ)< \lfloor m(n-m)\frac{a-1}{a}\cdot m \epell \rfloor+n\frac{a-1}{a}=(m+n)\frac{a-1}{a}$. So (IV) is proved.
\smallskip

(V)  Let $i\in [1,m(n-m)\frac{a-1}{a}]$. If $\ff^\circ$ satisfies the $i$-th collection of DSC, then $\ff^\circ$ satisfies the $(i-1)$-th collection of DSC. 

Proof of (V): assume $\ff^\circ$ satisfies the $i$-th collection of DSC. So 
$\text{supp}(F^\circ)\subseteq T_{a_{i},b_{i}}.$
By Conjecture~\ref{Jac_conj4}(iii), $\ff^\circ$ satisfies the $(i-1)$-th collection of DC, and we need to show that 
$\text{supp}(F^\circ)\subseteq %T_{\lfloor i m \epell \rfloor, \lfloor i m \epell \rfloor + \lfloor (i-1)(n-m) \epell  \rfloor
T_{a_{i-1},b_{i-1}}.$

Note that the $i$-th collection of DSC is the same as the $(i-1)$-th collection if $(a_i,b_i)=(a_{i-1},b_{i-1})$. Because of \eqref{xiyi}, there are two cases to consider:

Case 1: $(a_{i-1},b_{i-1})=(a_i,b_i)-(0,1)$. Then lattice points in $T_{a_{i},b_{i}} \setminus T_{a_{i-1},b_{i-1}}$ are one of the following form: 

-- Form (a): $(s,b_i)$ with $s\le a_i$,

-- Form (b): $(s,t)$ with $s+t=a_i+b_i$ and $0\le t\le b_i$. 

We first show that $\text{supp}(F^\circ)$ contains no points of Form (a). 
Since we assume $\ff^\circ$ satisfies the $(i-1)$-th collection of DC, the power $(x+1)^{j-(b_{i-1}-a_{i-1})}=(x+1)^{j-(b_{i}-a_{i})+1}$ divides $(F^\circ)_j^{(0,1)}$ for any $j$. In particular, take $j=b_i$, then $(x+1)^{a_{i}+1}$ must  divide $(F^\circ)_{b_i}^{(0,1)}$. However, $(1,0)\textrm{-deg}(F^\circ)_{b_i}^{(0,1)}\le a_i$. This forces 
$(F^\circ)_{b_i}^{(0,1)}=0$. So  $\text{supp}(F^\circ)$ contains no points of Form (a).  

Next we show that $\text{supp}(F^\circ)$ contains no points of Form (b). 
Since now $(a_i,b_i)\notin\textrm{Supp}(F^\circ)$, we have $(0,1)\textrm{-deg}(F^\circ)_{a_i+b_i}^{(1,1)}<b_i$; but on the other hand, by the $i$-th collection of DC we know that $(x+y)^{j-a_{i}}$ divides $(F^\circ)_j^{(1,1)}$, in particular let $j=a_i+b_i$, we have $(x+y)^{b_i}|(F^\circ)_{a_i+b_i}^{(1,1)}$. 
This forces $(F^\circ)_{a_i+b_i}^{(1,1)}=0$. So $\text{supp}(F^\circ)$ contains no points of Form (b). 

Hence we conclude that 
$\text{supp}(F^\circ)\subseteq T_{a_{i-1},b_{i-1}}$ in Case 1. 

Case 2: $(a_{i-1},b_{i-1})=(a_i,b_i)-(1,1)$. Then lattice points in $T_{a_{i},b_{i}} \setminus T_{a_{i-1},b_{i-1}}$ are one of the following form: 

-- Form (a): $(s,b_i)$ with $s\le b_i$, 

-- Form (b): $(s,t)$ with $s+t=a_i+b_i$ and $0\le t\le b_i$, 

-- Form (c): $(s,t)$ with $s+t=a_i+b_i-1$ and $0\le t\le b_i-1$.

First, we show that $\text{supp}(F^\circ)$ contains no points of Form (b). 
Since we assume $\ff^\circ$ satisfies the $(i-1)$-th collection of DC, the power $(x+y)^{j-a_{i-1}}=(x+y)^{j-a_{i}+1}$ divides $(F^\circ)_j^{(1,1)}$ for any $j$. In particular, take $j=a_i+b_i$, then $(x+y)^{b_{i}+1}$ must  divide $(F^\circ)_{a_i+b_i}^{(1,1)}$. However, $(0,1)\textrm{-deg}(F^\circ)_{a_i+b_i}^{(1,1)}\le b_i$. This forces 
$(F^\circ)_{a_i+b_i}^{(1,1)}=0$. So  $\text{supp}(F^\circ)$ contains no points of Form (b).  

Second, we show that $\text{supp}(F^\circ)$ contains no points of Form (a). 
Since now $(a_i,b_i)\notin\textrm{Supp}(F^\circ)$, we have $(1,0)\textrm{-deg}(F^\circ)_{a_i+b_i}^{(1,1)}<a_i$; but on the other hand, by the $i$-th collection of DC we know that $(x+1)^{j-(b_i-a_{i})}$ divides $(F^\circ)_j^{(0,1)}$, in particular let $j=b_i$, we have $(x+y)^{a_i}|(F^\circ)_{b_i}^{(0,1)}$. 
This forces $(F^\circ)_{b_i}^{(0,1)}=0$. So $\text{supp}(F^\circ)$ contains no points of Form (a). 

Third, we show that $\text{supp}(F^\circ)$ contains no points of Form (c). 
Since we assume $\ff^\circ$ satisfies the $(i-1)$-th collection of DC, the power $(x+y)^{j-a_{i-1}}=(x+y)^{j-a_{i}+1}$ divides $(F^\circ)_j^{(1,1)}$ for any $j$. In particular, take $j=a_i+b_i-1$, then $(x+y)^{b_{i}}$ must  divide $(F^\circ)_{a_i+b_i-1}^{(1,1)}$. However, since Form (a) is impossible, $(0,1)\textrm{-deg}(F^\circ)_{a_i+b_i-1}^{(1,1)}\le b_i-1$. This forces 
$(F^\circ)_{a_i+b_i-1}^{(1,1)}=0$. So  $\text{supp}(F^\circ)$ contains no points of Form (c).  

So we conclude that 
$\text{supp}(F^\circ)\subseteq T_{a_{i-1},b_{i-1}}$ in Case 2. 

This proves (V).

(VI) Note that by (II) and (IV), if $i=m(n-m)\frac{a-1}{a}$, then $\ff^\circ$ satisfies the $i$-th collection of DSC. Using this as the base case and applying a downward induction on $i$, we 
conclude that  $F^\circ$ satisfies the $0$-th collection of DSC. Then 
$\text{supp}(F^\circ)\subseteq T_{0,0}$, but actually $\text{supp}(F^\circ)\subseteq T_{0,-1}=\emptyset$ thanks to Lemma~\ref{multiple_adelta}. Therefore $F^\circ=0$.
%So $F^\circ=c$, a constant. If $c\neq0$, we can replace $\g^\circ$ by $\g^\circ+c$ and get a new $F^\circ=0$ which contradicts the minimal choice of $V^\circ \ (=\textrm{supp}(F^\circ)\ )$. Thus $F^\circ=0$.
\end{proof}

\section{The remainder vanishing conjecture}
In this section, we state the remainder vanishing conjecture. The statement may look complicated, but every object appearing here is directly computable.

\subsection{Set-up}

Fix $(a, b, m, n) \in \mathcal{Q}$. Define $\Delta$ to be the set of positive common divisors of $\frac{m}{a}$ and $\frac{n}{a}$, and fix $\delta\in\Delta$. 
Let $\mathcal{I}=\{i\in [0,m(n-m)]\, : \, \frac{i}{m}\in \mathbb{Z}\text{ or }\frac{i+1}{n-m}\in \mathbb{Z}\}$
and fix $i\in \mathcal{I}$.
%$\lfloor i(n-m) \epell  \rfloor >\lfloor (i-1)(n-m) \epell  \rfloor$ or $\lfloor (i+1) m \epell \rfloor> \lfloor i m \epell \rfloor$.

Define $u,L,\lowerE,\lowerF$ in two cases as below. Note that the two cases cannot hold simultaneously because of Lemma \ref{not_both}.  
\begin{equation}\label{uLetc1}
\left\{\begin{array}{ll}
u=0, 
L=x+1,  
\lowerE=\frac{n-m}{\delta a}, 
\text{ and }\lowerF= \lfloor i(n-m) \epell  \rfloor,&
\text{ if }\frac{i}{m}\in \mathbb{Z};\\
&\\
u=1, 
L=x+y, 
\lowerE=\frac{m}{\delta a}, 
\text{ and }\lowerF= \lfloor (i+1) m \epell \rfloor, &
\text{ if }\frac{i+1}{n-m}\in \mathbb{Z}.
\end{array}
\right.
\end{equation}
Let 
$$
R_2=\begin{cases}
\mathbb{C}[x^{\pm1},y^{\pm\frac{1}{m}}] &\textrm{ if } u=0;\\
\mathbb{C}[x^{\pm\frac{1}{n}},y^{\pm1}] &\textrm{ if } u=1.\\
\end{cases}
$$
Recall that in \eqref{aibi} we defined
$a_i=\lfloor (i+1) m \epell \rfloor$, $b_i=\lfloor (i+1) m \epell \rfloor + \lfloor i(n-m) \epell \rfloor$. 
Define
\begin{equation}\label{uLetc2}
\left\{\begin{array}{l}
w=(u,1)\in W;\\
d=um+n;\\
e=bd/a;\\
\mathfrak{m}=d+e-u-2;\\
\upperE=d/(\delta a);\\
\upperF=\left\{\begin{array}{ll} 
\lowerF + \frac{(un-um+m)(a-1)}{a}-1, \text{ if }i>\frac{m(n-m)(a-1)}{a};\\[1ex]
ua_i+b_i-1, \text{ if $i\le \frac{m(n-m)(a-1)}{a}$ and $\frac{a\delta\lfloor i(n-m) \epell  \rfloor}{n-m}=\frac{a\delta\lfloor (i+1) m \epell \rfloor}{m} \in \mathbb{Z}$};\\[1ex]
ua_i+b_i, \text{ otherwise};\end{array}\right.\\[1ex]
R_1=\mathbb{C}[\tilde{\Gamma}_0,\tilde{\Gamma}_1,...,\tilde{\Gamma}_{\upperE-1}, \tilde{x}_0,\tilde{x}_1,...,\tilde{x}_\upperF, \tilde{c}_1,\tilde{c}_2,...];\\
\tilde{F}^\circ=\sum_{j=\lowerF}^{\upperF} \tau^{j-\lowerF} \tilde{x}_{j} + \sum_{j=0}^{\lowerF-1} \tilde{x}_{j}\in R_1[\tau];\\[1ex]
\tilde{E}^\circ=\tau^{\upperE-\lowerE}  + \sum_{j=\lowerE}^{\upperE-1}  \tau^{j-\lowerE} \tilde{\Gamma}_{j}+ \sum_{j=0}^{\lowerE-1}   \tilde{\Gamma}_{j}\in R_1[\tau].
\end{array}\right.\end{equation}

\begin{example}\label{example2348}
Fix $(a,b,m,n)=(2,3,4,8)$, $\delta=1$, and $i=15$. Then $\frac{i+1}{n-m}\in \mathbb{Z}$. So $$\begin{array}{cccccccc}
u=1, &
d=12, & e=18, &
L=x+y, &
\lowerE=2, &
\upperE=6, &
\lowerF= 4, & \upperF=7.
\end{array}$$
\end{example}

Let $\text{Hom}(R_1[\tau],R_2)$ be the set of ring homomorphisms from $R_1[\tau]$ to $R_2$. 
For $w\in W$, we define the subset $\mathcal{S}^w$ of $\text{Hom}(R_1[\tau],R_2)$ by
$$\aligned
&\mathcal{S}^{(0,1)}=\{S\in \text{Hom}(R_1[\tau],R_2) \, :\, 
 S(\tau)= (x+1)y^{n/m} \};\\
&\mathcal{S}^{(1,1)}=\{S\in \text{Hom}(R_1[\tau],R_2) \, :\, 
 S(\tau)= x^{m/n}(x+y) \}.\\
\endaligned
$$

For $r\in\mathbb{Z}$, let $[r]_+=\max(0,r)$. This notation is commonly used for objects in the theory of cluster algebras.

Fix $E^\circ, F^\circ \in \mathbb{C}[x,y]$, $S\in \mathcal{S}^w$,    and 
$\g^\circ\in \mathbb{T}$ such that
\begin{equation}\label{mathfrakversion}\left\{\begin{array}{l}
S(\tilde{c}_\beta)\in \mathbb{C}\text{ for all }\beta\in[1,\mathfrak{m}];\\
F^\circ=S(\tilde{F}^\circ), w\text{-deg}(F^\circ)\le \upperF,\text{ and }(F^\circ)_j^w=S(\tau^{[j-\lowerF]_+} \tilde{x}_{j}) \text{ for }j\in [0,\upperF];\\
E^\circ=S(\tilde{E}^\circ), w\text{-deg}(E^\circ)\le \upperE,\text{ and }(E^\circ)_j^w=S(\tau^{[j-\lowerE]_+} \tilde{\Gamma}_{j})\text{ for }j\in [0, \upperE-1];\\
\text{deg}(\g^\circ)=\delta a.
\end{array}
\right.\end{equation}

Give a grading on $R_1[\tau]$ as follows: 
%$\deg(x)=u$, 
%$\deg(y)=1$, 
$$\deg(\tau)=\left\{
\begin{array}{ll}
n/m,&\text{ if }w=(0,1); \\
m/n+1,&\text{ if }w=(1,1),\\
\end{array}\right.
$$
$\deg(\tilde{\Gamma}_j)=j-[j-\lowerE]_+\deg(\tau)$, and 
$\deg(\tilde{x}_j)=j-[j-\lowerF]_+\deg(\tau)$. Note that  $\deg(\tau^{[j-\lowerE]_+} \tilde{\Gamma}_{j})=\deg(\tau^{[j-\lowerF]_+} \tilde{x}_{j})=j$.

\iffalse
{\color{red}Define the function (the homogenization map) $h:R_1[\tau^{\pm 1}]\longrightarrow R_1[\tau^{\pm1},t]$ by
$$
h(f)=\sum_{j\le\text{deg}(f) } f_j t^{\text{deg}(f)-j}.
$$
[There is a problem: monomials in $f$ may not have integer degree.] 
}
\fi

For any nonzero element $f\in R_1[\tau]$ which is $\mathbb{Z}$-graded (i.e. $f$ is of the form $f=\sum_{j\in\mathbb{Z}} f_j$), define $\deg(f)=\max\{j:f_j\neq0\}$, and define the homogenization map
$$
h(f)=\sum_{j\le\text{deg}(f) } f_j t^{\text{deg}(f)-j}\in R_1[\tau, t].
$$
Note that all terms in $\g^\circ(\tilde{E}^\circ) + \tilde{F}^\circ$ are of integer degrees, so we can apply the homogenization map and get $$H=h\big(\g^\circ(\tilde{E}^\circ) + \tilde{F}^\circ\big).$$

For any $y_0,\dots,y_n\in R_1[\tau^{\pm 1}]$ and for any  $A=r/s$ with $r\in\mathbb{Z}$ and $s\in\mathbb{Z}_{>0}$,   
we have the following identity in the ring $R_1[\tau^{\pm 1},y_0^{\pm 1/s}][[t]]$ (where we fix a choice of $y_0^{1/s}$): 
\begin{equation}\label{eq:xA}
(y_0+y_1t+\cdots+y_{n}t^n)^A=\sum_{v_1,\dots,v_n\in\mathbb{Z}_{\ge0}}\frac{A(A-1)\cdots(A-\sum_{j=1}^n v_j) }
{\prod_{j=1}^{n} v_j!} y_0^{A-\sum_{j=1}^n v_j}y_1^{v_1}\cdots y_n^{v_n}t^{v_1+\cdots+nv_n}.
\end{equation}

For any $\widetilde{F}\in R_1[\tau^{\pm1}][[t]]$, denote  
$$[\widetilde{F}]_{t^j}=\textrm{ the coefficient of $t^j$ in $\widetilde{F}$},$$
which is an element in $R_1[\tau^{\pm1}]$. 

For any $\widetilde{F}\in R_1[\tau^{\pm1}]$, denote  
$$[\widetilde{F}]_{\tau^j}=\textrm{ the coefficient of $\tau^j$ in $\widetilde{F}$},$$
which is an element in $R_1$.

Let $r=\gcd(m,n)$. For each $\mu\in[0,\mathfrak{m}]$ and each $B\subseteq\{\beta\in[1,\mathfrak{m}] \,:\, r(e-\beta)/d \in\mathbb{Z}\}$, 
%{\color{red} Note that $r$ is not defined yet. In our setting, if $w=(0,1)$, then $F_d=(x+1)^my^n=(polynomial)^{gcd(m,n)}$, so $r=gcd(m,n)$; if $w=(1,1)$, then $F_d=(x+y)^nx^m$, again $r=gcd(m,n)$.} 
let 
\begin{equation}\label{mathfrakversion2}
\tilde{G}_{B,e-\mu} =  [H^{\frac{e}{d}}]_{t^{\mu}} + \sum_{\beta\in B\cap[1,\mu]}  \tilde{c}_\beta [H^{\frac{e-\beta}{d}}]_{t^{\mu-\beta}} \in R_1[\tau^{\pm1}],
\end{equation}
and
let $\mathfrak{j}(\mu)=\min\{\mathfrak{j}\in \mathbb{Z}_{\ge 0} \, : \, [\tilde{G}_{B,e-\mu}]_{\tau^{-j}}=0\text{ for all }j> \mathfrak{j}\}$. For $k\in [1,\mathfrak{j}(\mu)]$, define $\tilde{G}_{B,e-\mu, k}$ by 
$$
\tilde{G}_{B,e-\mu, k} =\sum_{j=\mathfrak{j}(\mu)+1-k}^{\mathfrak{j}(\mu)}\tau^{\mathfrak{j}(\mu)-j}[\tilde{G}_{B,e-\mu}]_{\tau^{-j}} \in  R_1[\tau]. 
$$
If $B=\{\beta\in[1,\mathfrak{m}] \,:\, r(e-\beta)/d \in\mathbb{Z}\}$, then we simply denote $\tilde{G}_{B,e-\mu}$ by $\tilde{G}_{e-\mu}$, and $\tilde{G}_{B,e-\mu, k}$ by $\tilde{G}_{e-\mu, k}$.

\subsection{The statement of the remainder vanishing conjecture}
For $k\in\mathbb{Z}_{\ge 0}$, define the following natural projection as a ring homomorphism:
 $$
P_k:R_2\longrightarrow R_2/ L^{k}.$$ 
\iffalse 
$$
\begin{cases}
P_k:\mathbb{C}[x^{\pm1},y^{\pm\frac{1}{m}}]\longrightarrow 
\mathbb{C}[x^{\pm1},y^{\pm\frac{1}{m}}]/ L^{k}&\text{ if }w=(0,1);\\
P_k:\mathbb{C}[x^{\pm\frac{1}{n}},y^{\pm1}]\longrightarrow 
\mathbb{C}[x^{\pm\frac{1}{n}},y^{\pm1}]/ L^{k}&\text{ if }w=(1,1).\\
\end{cases}
$$
\fi
Note that the codomain of $P_1$ is an integral domain, in particular,

-- if $P_1(fg)=0$, then $P_1(f)=0$ or $P_1(g)=0$; 

-- if $P_1(f^j)=0$ for some integer $j>0$, then $P_1(f)=0$. 

\noindent Indeed, this follows immediately from the fact that $(x+1)$ is irreducible in $\mathbb{C}[x^{\pm1},y^{\pm\frac{1}{m}}]$, and $x+y$ is irreducible in $\mathbb{C}[x^{\pm\frac{1}{n}},y^{\pm1}]$. 
%https://math.stackexchange.com/questions/543004/if-r-is-a-ufd-then-rx-x-1-is-a-ufd

%Since $\GG_{e-\mu}=S^{w}(\tilde{G}_{e-\mu})$ and $$\GG_{e-\mu}\text{ is }\left\{\begin{array}{ll}\text{a polynomial }&\text{ for }\mu\in [0,\mathfrak{m}];\\
%\text{equal to }0&\text{ for }\mu\in [e+1,\mathfrak{m}], \end{array} \right.$$ 
%we get the system of equations as in $\eqref{soeq}$.

\begin{definition}\label{supported}
Fix $(a,b,m,n)\in\mathcal{Q}$, $\delta\in \Delta$, and $i\in\mathcal{I}$, which will determine  the objects appearing in \eqref{uLetc1} and \eqref{uLetc2}. Fix $S, E^\circ, F^\circ, \g^\circ$ so that \eqref{mathfrakversion} is satisfied. 
We say that  a subset $B$ of $\{\beta\in[1,\mathfrak{m}] \,:\, r(e-\beta)/d \in\mathbb{Z}\}$  \emph{is supported with respect to }$(S,E^\circ,F^\circ,\g^\circ)$ if
\begin{equation}\label{soeq1}
\left\{\begin{array}{ll} P_{k}( S(\tilde{G}_{B,e-\mu, k}))=0 &\text{ for }\mu\in [0,\mathfrak{m}]\text{ and }k\in [1,\mathfrak{j}(\mu)];\\
S(\tilde{G}_{B,e-\mu})=0  & \text{ for } \mu\in [e+1,\mathfrak{m}].\\
\end{array} \right.
\end{equation}
\end{definition}
Note that if $B'\subseteq B$ and $B'$ is supported with respect to $(S,E^\circ,F^\circ,\g^\circ)$, then there exists $S'\in \mathcal{S}^w$ such that $S'(\tilde{c}_\beta)=0$ for every $\beta\in B\setminus B'$, and that each of $B'$ and $B$ is supported with respect to $(S',E^\circ,F^\circ,\g^\circ)$.

Denote $x_j=S(\tilde{x}_j)$ for $0\le j\le \upperF$, and $\Gamma_j=S(\tilde{\Gamma}_j)$ for $0\le j\le \upperE-1$.

In Section~\ref{EtoDsection}, we will show that the following conjecture implies the Jacobian conjecture.

\begin{conjecturealpha}\label{Jac_conj5} Assume the hypotheses as in Definition~\ref{supported}. Fix $\ell\in [\mathrm{u}_F,\mathrm{v}_F]$.
Suppose that
\begin{equation}\label{soeq}
B\text{ is supported with respect to }(S,E^\circ,F^\circ,\g^\circ)\text{ and }P_1(x_j)=0\text{ for } j\in [\ell+1,\mathrm{v}_F].
\end{equation}
If $E^\circ$ is a principal polynomial, then either $P_1(x_\ell)=0$ or there exists a proper subset of $B$ which is supported with respect to $(S,E^\circ,F^\circ,\g^\circ)$. 
%More precisely, we have the following:\\
%(a) If $i> (... upper right corner)$ and $\lfloor (i+1) m \epell \rfloor :  \lfloor i(n-m) \epell  \rfloor \neq m : (n-m)$, then $P_1(x_h)=0$.

%(b) Suppose $i\le (... upper right corner)$. If $\lfloor (i+1) m \epell \rfloor :  \lfloor i(n-m) \epell  \rfloor = m : (n-m)$ and $\text{deg}(f)\not |\frac{\lfloor (i+1) m \epell \rfloor}{m}\text{deg}(P^\circ)$ for any $f\in\mathcal{P}(P^\circ)$, then $P_1(x_h)=0$.
\end{conjecturealpha}

This conjecture will be called the  \emph{remainder vanishing conjecture}. Note that when $\ell=\upperF$, the condition ``$P_1(x_j)=0$ for $j\in [\ell+1,\upperF]$'' is vacuous. Also note that Conjecture E has the following consequence by applying it iteratively:

\noindent {\bf Consequence of Conjecture \ref{Jac_conj5}.} %{\it  Assume the hypotheses as in Definition~\ref{supported}. Suppose that $E^\circ$ is a principal polynomial and $B$ is a subset of $\{\beta\in[1,\mathfrak{m}] \,:\, r(e-\beta)/d \in\mathbb{Z}\}$ supported with respect to $(S,E^\circ,F^\circ,\g^\circ)$. Then $P_1(x_\ell)=0$ for all $\ell\in[\mathrm{u}_F,\mathrm{v}_F]$.}
 {\it  Assume the same hypotheses. Fix $\ell\in [\mathrm{u}_F,\mathrm{v}_F]$. Suppose that $E^\circ$ is a principal polynomial, $B$ is a subset of $\{\beta\in[1,\mathfrak{m}] \,:\, r(e-\beta)/d \in\mathbb{Z}\}$ supported with respect to $(S,E^\circ,F^\circ,\g^\circ)$, and $P_1(x_j)=0\text{ for } j\in [\ell+1,\mathrm{v}_F]$. Then $P_1(x_\ell)=0$.}
%Suppose that $P_1(x_h)\neq 0$. Then a proper subset, say $B'$, of $B$ is supported with respect to $(S,E^\circ,F^\circ,\g^\circ)$. After replacing $B$ with $B'$, we apply Conjecture~\ref{Jac_conj5}. Then a proper subset of $B'$ is supported with respect to $(E^\circ,F^\circ)$. Since $B$ is finite, iterating this argument eventually yields a contradiction. Therefore  $P_1(x_h)$ must be equal to $0$.
\smallskip

%\begin{example}

%\end{example}

%\begin{remark}
%The definition of $\upperF$, which is given in \eqref{uLetc2}, is crucial here. Because if $\upperF$ were replaced with a number which is greater than $\upperF$, then it would become very hard to solve. 
%\end{remark}

There is a systematic way to solve $\eqref{soeq}$. In \cite{GHLL3}, we give some examples to illustrate how to solve it. In the forthcoming paper(s) including \cite{GHLL4}, we plan to prove the  remainder vanishing conjecture.

\section{The implication from the remainder vanishing conjecture to Conjecture~\ref{Jac_conj4}}\label{EtoDsection}

%In what follows, the binomial coefficient ${A\choose B}$  is defined by \[{A\choose B}:={A(A-1)\cdots(A-B+1)\over B!}\] for any real number $A$ and any nonnegative integer $B$.

 We recall the generalized Magnus' formula from \cite{HLLN}.
%Recall that $\cR=K[x,y]$, where $K$ is an algebraically closed field of characteristic 0. 

\begin{theorem}[\cite{HLLN}] \label{prop:Magnus_formula_equivalence}
Suppose $[F, G] \in  \mathbb{C}$. For any direction $w=(u,v) \in W$, let $d=w\text{-}\deg(F_+)$ and $e=w\text{-}\deg(G_+)$.  Assume $d>0$. Write the $w$-homogeneous  decompositions $F=\sum_{i\leq d} F_i$ and $G=\sum_{i\leq e} G_i$. 
Let $r\in\mathbb{Z}_{>0}$ be the largest integer such that $F_d^{1/r}\in  \mathbb{C}[x,y]$.\footnote{There is some ambiguity of the notation $F_d^{1/r}$ since it is unique up to an $r$-th root of unity, so we need to fix a choice of $F_d^{1/r}$. In this paper, there is often a natural choice of $F_d^{1/r}$ by requiring its certain coefficient to be $1$.} 
Then there exists a unique sequence of constants $c_0,c_1,\dots,c_{d+e-u-v-1}\in  \mathbb{C}$ such that $c_0\neq 0$ and 
\begin{equation}\label{Magnus_formula_equivalence}
G_{e-\mu} = \sum_{\beta=0}^{\mu} c_\beta [h(F)^{\frac{e-\beta}{d}}]_{t^{\mu-\beta}}
\end{equation}
for every integer $\mu\in\{0,1,...,d+e-u-v-1\}$. 
Moreover, $c_\beta=0$ if $r(e-\gamma)/d\notin\mathbb{Z}$. 
\end{theorem}

Note that the last condition on $c_\gamma$ implies that every nonzero summand appearing on the right side of \eqref{Magnus_formula_equivalence} is in $\mathcal{R}[F_d^{-1/r}]$, so must be a rational function.

\begin{lemma} 
Conjecture~\ref{Jac_conj5} implies Conjecture~\ref{Jac_conj4}.
\end{lemma}
\begin{proof}
Let $(a,b,m,n)\in\mathcal{Q}$ and $\delta\in \Delta$. Let $i\in\mathcal{I}$, which will determine the objects appearing in \eqref{uLetc1} and \eqref{uLetc2}. Let the polynomials $E^\circ, F^\circ$, and $\g^\circ$ be defined as in Definition \ref{def:Ecirc} and \ref{def:VF}. In particular, $E^\circ$ is a principal polynomial. We shall prove the lemma in three steps.

\smallskip

\noindent (I) We show that Conjecture~\ref{Jac_conj5} implies Conjecture~\ref{Jac_conj4}{\rm(i)}. 

Let $w=(0,1)$ and $i\in [m(n-m)(a-1)/a+1, m(n-m)]$. 
Suppose that $\ff^\circ$ satisfies the $i$-th collection of $w$-DC. Then Lemma~\ref{Div_to_limit1} implies $w\text{-deg}(\ff^\circ)< \lfloor i(n-m) \epell  \rfloor+m\frac{a-1}{a}$. (Note that the right side of the inequality is $\lfloor \frac{i}{m} \rfloor+m\frac{a-1}{a}=\upperF+1$, so the inequality can be rewritten as ``$w\text{-deg}(\ff^\circ)\le \upperF$''.)  We need to show that  $\ff^\circ$ satisfies the $(i-1)$-th collection of $w$-DC.

The $i$-th collection of $w$-DC implies that $(F^\circ)^{w}_{j}$ is divisible by $L^{j-\lowerF}$ for each $j\in[\lowerF,\upperF]$. 
 Then there exist  $x^*_0,...,x^*_\upperF\in\mathcal{R}$ such that 
$$
F^\circ = \sum_{j=0}^{\upperF}L^{[j-\lowerF]_+} x^*_{j} =\sum_{j=\lowerF}^{\upperF} L^{j-\lowerF} x^*_{j} + \sum_{j=0}^{\lowerF-1} x^*_{j},
$$
where $L^{[j-\lowerF]_+} x^*_{j} \in \mathcal{R}_j^w$  for $j\in [0,\upperF]$.

Assume that $\lfloor i(n-m) \epell  \rfloor \neq \lfloor (i-1)(n-m) \epell  \rfloor$ %and  $\lfloor (i+1) m \epell \rfloor = \lfloor i m \epell \rfloor$,
(otherwise, the $i$-th collection of $(0,1)$-DC is the same as the $(i-1)$-th collection of $(0,1)$-DC, so there is nothing to prove).  
We need to show that $x^*_j$ is divisible by $L$ for $j\in[\lowerF,\upperF]$.

By Lemma~\ref{Ecirctopdegpieces} and Proposition~\ref{PQdivisible}(i), there exist  $\Gamma^*_0,...,\Gamma^*_{\upperE-1}\in\mathcal{R}$ such that 
$$
E^\circ = L^{\upperE-\lowerE}y^{\frac{n}{a\delta}} +  \sum_{j=0}^{\upperE-1}L^{[j-\lowerF]_+} \Gamma^*_{j} = L^{\upperE-\lowerE}y^{\frac{n}{a\delta}} +  \sum_{j=\lowerE}^{\upperE-1} L^{j-\lowerE} \Gamma^*_{j} + \sum_{j=0}^{\lowerE-1} \Gamma^*_{j},
$$
with $
L^{[j-\lowerE]_+} \Gamma^*_{j} \in \mathcal{R}_j^w $ for $j\in [0, \upperE-1]$ .

%We remember that $F=\g^\circ(E^\circ) + F^\circ$ as in Section ??. 

By Theorem~\ref{prop:Magnus_formula_equivalence}, there exists a unique sequence of constants $c_0,c_1,\dots,c_{\mathfrak{m}}\in \mathbb{C}$ such that $c_0=1$ and 
$$
G_{e-\mu} = \sum_{\gamma=0}^{\mu} c_\gamma [h(\g^\circ(E^\circ) + F^\circ)^{\frac{e-\gamma}{d}}]_{t^{\mu-\gamma}}.
$$
Let $S\in \mathcal{S}^w$ be a morphism satisfying $S(\tilde{c}_\beta)=c_\beta$ for $\beta\in[1,\mathfrak{m}]$, $S(\tau^{[j-\lowerE]_+} \tilde{\Gamma}_{j})= L^{[j-\lowerE]_+} \Gamma^*_{j} $ for $j\in [0, \upperE-1]$, and $S(\tau^{[j-\lowerF]_+}  \tilde{x}_{j}) = L^{[j-\lowerF]_+} x^*_{j}$ for $j\in [0,\upperF]$. Then $S, E^\circ, F^\circ$, and $\g^\circ$  satisfy \eqref{mathfrakversion}.
Then \eqref{mathfrakversion2} implies that $S(\tilde{G}_{B,e-\mu})={G}_{e-\mu}$ for $B=\{\beta\in[1,\mathfrak{m}] \,:\, r(e-\beta)/d \in\mathbb{Z}\}$.
Since $\GG_{e-\mu}=S(\tilde{G}_{B,e-\mu})$ and $$\GG_{e-\mu}\text{ is }\left\{\begin{array}{ll}\text{a polynomial, }&\text{ for }\mu\in [0,e];\\
\text{equal to }0, &\text{ for }\mu\in [e+1,\mathfrak{m}], \end{array} \right.$$ 
the set  $B$  is supported with respect to $(S,E^\circ,F^\circ,\g^\circ)$.

Fix any $h\in [\lowerF,\upperF]$, and suppose that $P_1(x_j)=0\text{ for } j\in [h+1,\upperF]$.  Since $E^\circ$ is a principal polynomial, by Consequence of Conjecture~\ref{Jac_conj5}, we conclude that $P_1(x_h)=0$. Using this argument iteratively for $h=\upperF,\upperF-1,\dots,\lowerF$ in decreasing order, we conclude that $P_1(x_j)=0$ for all $j\in [\lowerF,\upperF]$. This implies that  $x^*_{j}$ is divisible by $L$ for $j\in[\lowerF,\upperF]$, because $S(\tilde{x}_{j})$ is the product of a monomial in $y^{\pm\frac{1}{mn}}$ and a polynomial in $x$ and $y$
by condition \eqref{mathfrakversion}.

%This completes the proof of (I).
\smallskip

\noindent (II) We show that Conjecture~\ref{Jac_conj5} and Conjecture~\ref{Jac_conj4}(i) together  imply Conjecture~\ref{Jac_conj4}(ii). 
The proof is similar to (I) and we will focus on the difference. 

Let $w=(1,1)$ and $i\in [m(n-m)(a-1)/a, m(n-m)]$. 
Suppose that $\ff^\circ$ satisfies the $i$-th collection of $w$-DC. Then Lemma~\ref{Div_to_limit2} implies  $w\text{-deg}(\ff^\circ)<\lfloor (i+1)m\epell\rfloor+n\frac{a-1}{a}$ (where the inequality can again be rewritten as ``$w\text{-deg}(\ff^\circ)\le \upperF$'').   We need to show that  $\ff^\circ$ satisfies the $(i-1)$-th collection of $(0,1)$-DC.
Write
$F^\circ = \sum_{j=0}^{\upperF}L^{[j-\lowerF]_+} x^*_{j}$
and assume $\lfloor (i+1) m\epell  \rfloor \neq \lfloor im\epell  \rfloor$. 
We need to show that $x^*_j$ is divisible by $L$ for $j\in[\lowerF,\upperF]$. The rest is similar to (I). Note that we need to assume Conjecture~\ref{Jac_conj4}(i) in order  to apply Proposition~\ref{PQdivisible}(ii).

\noindent (III) We show that Conjecture~\ref{Jac_conj5} and Conjecture~\ref{Jac_conj4}(i)(ii) together  imply Conjecture~\ref{Jac_conj4}(iii). 
The proof is similar to (I)(II) and we will focus on the difference. 

Let $i\in[1,\frac{m(n-m)(a-1)}{a}]$, and suppose that $\ff^\circ$ satisfies the $i$-th collection of DSC. 
Lemma~\ref{not_both} implies that 
there does not exist $i$ such that $\lfloor i(n-m) \epell  \rfloor \neq\lfloor (i-1)(n-m) \epell  \rfloor$ and $\lfloor (i+1) m \epell \rfloor\neq \lfloor i m \epell \rfloor$. 

Case (i): if $\lfloor i(n-m) \epell  \rfloor \neq\lfloor (i-1)(n-m) \epell  \rfloor$. In this case, $\lfloor (i+1) m \epell \rfloor =\lfloor i m \epell \rfloor$,  thus $\ff^\circ$ satisfies the $(i-1)$-th collection of $(1,1)$-DC because it is the same as the $i$-th collection. Let $w=(0,1)$, by the $i$-th collection of DSC we can write 
$F^\circ = \sum_{j=0}^{\upperF}L^{[j-\lowerF]_+} x^*_{j}$ where $\upperF=b_i$. Moreover, when 
$\frac{a\delta\lfloor i(n-m) \epell  \rfloor}{n-m}=\frac{a\delta\lfloor (i+1) m \epell \rfloor}{m} \in \mathbb{Z}$, 
we have 
$\text{supp}(F^\circ)\subseteq T_{a_i,b_i-1}$ by Lemma \ref{multiple_adelta}, so we can set $\upperF=b_i-1$. The rest is similar to (I).

Case (ii): if $\lfloor (i+1) m \epell \rfloor \neq \lfloor i m \epell \rfloor$. In this case, $\lfloor i(n-m) \epell  \rfloor =\lfloor (i-1)(n-m) \epell  \rfloor$,  thus $\ff^\circ$ satisfies the $(i-1)$-th collection of $(0,1)$-DC because it is the same as the $i$-th collection. Let $w=(1,1)$, by the $i$-th collection of DSC we can write 
$F^\circ = \sum_{j=0}^{\upperF}L^{[j-\lowerF]_+} x^*_{j}$ where $\upperF=a_i+b_i$. Moreover, when 
$\frac{a\delta\lfloor i(n-m) \epell  \rfloor}{n-m}=\frac{a\delta\lfloor (i+1) m \epell \rfloor}{m} \in \mathbb{Z}$, 
we have 
$\text{supp}(F^\circ)\subseteq T_{a_i,b_i-1}$ by Lemma \ref{multiple_adelta}, so we can set $\upperF=a_i+b_i-1$. The rest is the same as (II). 

This completes the proof. 
\end{proof}

%\begin{example}
%$g=z^2+u$ with $u\in\mathbb{C}$.
%\end{example}

\section{Appendix}\label{appendix}
In \cite{GGV14}, C. Valqui, J. A. Guccione and J. J. Guccione gave another conjecture which is equivalent to the Jacobian conjecture. Here we  slightly rephrase \cite[Theorem 1.9]{GGV14}. 

\begin{conjecture}\label{GGV_conj}
Let $a,b\in\mathbb{Z}_{>0}$ be such that $a\not|b$ and $b\not|a$. Let $\lambda_i\in \mathbb{C}$ for $i\in [0,a+b-2]$, with $\lambda_0=1$. Suppose that $F,G\in\mathcal{R}$ and $C,P\in\mathbb{C}[y]((x^{-1}))$ satisfy the following:

\noindent (1) $C$ has the form
$$
C=x + C_{-1}x^{-1} + C_{-2}x^{-2} + \cdots\text{ with each }C_{-i}\in \mathbb{C}[y];
$$

\noindent (2) $\text{deg}(C)=1$, and $\text{deg}(P)\leq 2-a$;

\noindent (3) if $\text{deg}(P)=2-a$ then $P_+=x^{1-a}y$, where $P_{+}$ is taken with respect to the $(1,0)$-grading;

\noindent (4) $C^a=G$ and $F=\sum_{i=0}^{a+b-2}\lambda_iC^{b-i} + P$;

\noindent (5) $[F,G]\in  \mathbb{C}$.

Then $[F,G]=0$.
\end{conjecture}

As suggested by Christian Valqui, it would be interesting to see if the generalized Magnus' formula could help solve the systems of polynomial equations obtained from the condition (4). 
%In particular, if you could solve the case a=2, b=3 (in your language), that would eliminate many small cases.

\end{document}